\documentclass[reqno,11pt]{amsart}

\usepackage{amssymb,color,hyperref,mathrsfs,stmaryrd}
\usepackage{amsmath,mathabx}

\usepackage[usenames,dvipsnames]{xcolor}

\usepackage[curve,matrix,arrow]{xy}

\numberwithin{equation}{section}
\newtheorem{theorem}{Theorem}[section]
\newtheorem{lemma}[theorem]{Lemma}
\newtheorem{hypo}[theorem]{Hypothesis}

\newtheorem{Th}{Theorem}
\newtheorem{Prop}[Th]{Proposition}

\newtheorem{corollary}[theorem]{Corollary}

\theoremstyle{definition}

\newtheorem{remark}[theorem]{Remark}

\newtheorem{ex}[theorem]{Example}

\newcommand{\F}{\mathcal{F}}
\newcommand{\E}{\mathcal{E}}
\newcommand{\G}{\mathcal{G}}

\newcommand{\X}{\mathcal{X}}
\renewcommand{\L}{\mathcal{L}}
\newcommand{\mD}{\mathcal{D}}

\newcommand{\tE}{\tilde{\mathcal{E}}}

\newcommand{\tS}{\tilde{S}}

\newcommand{\tH}{\tilde{\mathcal{H}}}
\newcommand{\M}{\mathcal{M}}
\newcommand{\N}{\mathcal{N}}
\newcommand{\K}{\mathcal{K}}
\newcommand{\D}{\mathbf{D}}

\renewcommand{\H}{\mathcal{H}}

\newcommand{\Hom}{\operatorname{Hom}}
\newcommand{\Aut}{\operatorname{Aut}}

\newcommand{\Syl}{\operatorname{Syl}}

\newcommand{\ov}{\overline}

\newcommand{\W}{\mathbf{W}}

\newcommand{\hyp}{\mathfrak{hyp}}
\newcommand{\Comp}{\operatorname{Comp}}

\newcommand{\hH}{\hat{\mathcal{H}}}

\newcommand{\fC}{\mathfrak{C}}

\newcommand{\Y}{\mathcal{Y}}
\def \<{\langle }
\def \>{\rangle }
\newcommand{\norm}{\trianglelefteq}
\newcommand{\subn}{{\,\norm\norm\,}}
\renewcommand{\unlhd}{\norm}
\newcommand{\la}{\<\!\<}
\newcommand{\ra}{\>\!\>}

\renewcommand{\phi}{\varphi}

\title{On Wielandt's Join Theorem for fusion systems and localities}

\author[E.~Henke]{Ellen Henke}

\begin{document}

\begin{abstract}
Saturated fusion systems are categories generalizing important aspects of conjugacy of $p$-subgroups in finite groups. It was shown by Chermak that there are group-like structures called regular localities associated to saturated fusion systems. Both the theory of fusion systems and the theory of regular localities are developed in analogy to the theory of finite groups. In this paper we focus on a classical theorem of Wielandt, which states that any two subnormal subgroups of a finite group $G$ generate a subnormal subgroup of $G$. We prove versions of this theorem for regular localities and for fusion systems. Along the way we prove also a purely group-theoretical result which may be of independent interest.
\end{abstract}

\thanks{\textit{Author affiliation.} Institut f{\"u}r Algebra, Fakult{\"a}t Mathematik, Technische Universit{\"a}t Dresden, 01062 Dresden, Germany\\
\textit{Email.} ellen.henke@tu-dresden.de\\
\textit{ORCID Id.} 0000-0002-9759-3708}


\maketitle

\section{Introduction}

Helmut Wielandt proved in his habilitation thesis \cite{Wielandt:1939} that any two subnormal subgroups $H_1$ and $H_2$ of a finite group $G$ generate a subnormal subgroup of $G$, a theorem which is often referred to as Wielandt's Join Theorem. It is moreover known that $\<H_1,H_2\>\cap S=\<H_1\cap S,H_2\cap S\>$ for any Sylow $p$-subgroup $S$ of $G$; as far as we are aware, this is a Lemma due to Ulrich Meierfrankenfeld (see Lemma~\ref{L:Meierfrankenfeld}). In the present paper we prove versions of these results for regular localities and for fusion systems. The reader is referred to Part I of \cite{Aschbacher/Kessar/Oliver:2011} for an introduction to the theory of fusion systems.

\smallskip

Localities are group--like structures associated to saturated fusion systems which were introduced by Chermak \cite{Chermak:2013,Chermak:2015}. Roughly speaking, a locality is a triple $(\L,\Delta,S)$, where $\L$ is a \emph{partial group} (i.e. a set $\L$ together with an ``inversion'' and a ``partial multiplication'' which is only defined on certain words in $\L$), $S$ is a  ``Sylow $p$-subgroup'' of $\L$, and $\Delta$ is a set of subgroups of $S$ subject to certain axioms. Given a partial group $\L$,  there are natural notions of \emph{partial normal subgroups} of $\L$ and of \emph{partial subnormal subgroups} of $\L$. A nice theory of partial subnormal subgroups and of components can be developed for \emph{regular localities}, which are special kinds of localities introduced by Chermak \cite{ChermakIII} (see also \cite{Henke:Regular}). It turns out that partial subnormal subgroups of regular localities form regular localities. The existence and uniqueness of centric linking systems implies that there is an essentially unique regular locality associated to every saturated fusion system.

\smallskip

Wielandt's Join Theorem and Meierfrankenfeld's Lemma have the following very natural translation to regular localities. The proof uses Wielandt's Join Theorem and Meierfrankenfeld's Lemma for groups.

\begin{Th}\label{T:RegularLocalities}
Let $(\L,\Delta,S)$ be a regular locality. Fix moreover two partial subnormal subgroups $\H_1$ and $\H_2$ of $\L$. Then $\<\H_1,\H_2\>$ is a partial subnormal subgroup of $\L$ with
\[\<\H_1,\H_2\>\cap S=\<\H_1\cap S,\H_2\cap S\>.\]
\end{Th}

Theorem~\ref{T:RegularLocalities} is restated and proved in Theorem~\ref{T:c}. 

\smallskip

Finding a formulation of Wielandt's Join Theorem for fusion systems is slightly more tricky. As we show in Example~\ref{E:1}, the subsystem generated by two subnormal subsystems of a saturated fusion system $\F$ is in some cases not even saturated (and is thus in particular not a subnormal subsystem of $\F$). 
However, we are able to prove the following result.

\begin{Th}\label{T:FusionSystems}
Let $\F$ be a saturated fusion system, and suppose that $\E_1$ and $\E_2$ are subnormal subsystems of $\F$ over $S_1$ and $S_2$ respectively. Then there is a (with respect to inclusion) smallest saturated subsystem $\la\E_1,\E_2\ra$ of $\F$ in which $\E_1$ and $\E_2$ are subnormal. The subsystem $\<\!\<\E_1,\E_2\>\!\>$ is a subnormal subsystem of $\F$ over $\<S_1,S_2\>$.
\end{Th}

The reader is referred to Theorem~\ref{T:FusionSystemsn} for more information. It should be pointed out that
there is in many cases not a unique
smallest saturated subsystem (or a unique smallest subnormal subsystem) of $\F$ containing $\E_1$ and $\E_2$ (cf. Example~\ref{E:1}).

\smallskip

If $\E_1$ and $\E_2$ are \emph{normal} subsystems of $\F$, then we have introduced a product subsystem $\E_1\E_2$ already in a joint paper with Chermak \cite[Theorem~C]{Chermak/Henke}. In some sense we generalize this construction above, as it turns out that $\la \E_1,\E_2\ra=\E_1\E_2$ if $\E_1$ and $\E_2$ are normal subsystems. Theorem~\ref{T:FusionSystems} implies thus that $\E_1\E_2$ is the smallest saturated subsystem of $\F$ in which $\E_1$ and $\E_2$ are subnormal. Thereby, we also obtain a new characterization of $\E_1\E_2$.

\smallskip

Let us now say a few words about the proof of Theorem~\ref{T:FusionSystems}. As mentioned before, there is an essentially unique regular locality associated to every saturated fusion system. Moreover, if $(\L,\Delta,S)$ is a regular locality over a saturated fusion system $\F$, then it was  shown by Chermak and the author of this paper \cite[Theorem~F]{Chermak/Henke} that there is a natural one-to-one correspondence between the partial subnormal subgroups of $\L$ and the subnormal subsystems of $\F$. Therefore, if $\E_1$ and $\E_2$ are two subnormal subsystems of $\F$, then Theorem~\ref{T:RegularLocalities} can be used to show that there is a smallest subnormal subsystem $\E$ of $\F$ in which $\E_1$ and $\E_2$ are subnormal. Showing that $\E$ is the smallest \emph{saturated subsystem} in which $\E_1$ and $\E_2$ are subnormal is however still somewhat difficult. The essential ingredient in the proof of this property is the following group-theoretical result.

\begin{Th}\label{T:Groups}
Let $G$ be a finite group and $S\in\Syl_p(G)$. Let $H_1$ and $H_2$ be two subnormal subgroups of $G$. Then $T:=\<H_1,H_2\>\cap S=\<H_1\cap S,H_2\cap S\>$ and
\[\F_T(\<H_1,H_2\>)=\<\;\F_T(\<H_1,T\>),\;\F_T(\<H_2,T\>)\;\>.\]
\end{Th}

We feel that Theorem~\ref{T:Groups} might be of independent interest to finite group theorists. Therefore, we seek to keep the proof of this theorem elementary (cf. Subsection~\ref{SS:GroupFusion}). Theorem~\ref{T:Groups} can also be used to show the following Proposition (where we use the notation introduced in Theorem~\ref{T:FusionSystems}).

\begin{Prop}\label{P:GroupFusionBracket}
Let $G$ be a finite group and $S\in\Syl_p(G)$. Suppose $H_1$ and $H_2$ are subnormal subgroups of $G$. Then
\[\F_{\<H_1,H_2\>\cap S}(\<H_1,H_2\>)=\la\F_{H_1\cap S}(H_1),\F_{H_2\cap S}(H_2)\ra.\]
\end{Prop}

This says in particular that the fusion system $\F_{\<H_1,H_2\>\cap S}(\<H_1,H_2\>)$ is determined by the fusion systems $\F_S(G)$, $\F_{H_1\cap S}(H_1)$ and $\F_{H_2\cap S}(H_2)$.

\subsubsection*{Organization of the paper} We start by proving the basic group-theoretical results in Section~\ref{S:Groups}. More precisely, Meierfrankenfeld's Lemma is proved in Subsection~\ref{SS:Meierfrankenfeld} and building on that, Theorem~\ref{T:Groups} is proved in Subsection~\ref{SS:GroupFusion}. After summarizing some background on partial groups and localities in Section~\ref{S:PartialLocalities},  Theorem~\ref{T:RegularLocalities} and some related results are then proved in Section~\ref{S:Regular}. Finally, Theorem~\ref{T:FusionSystems} and Proposition~\ref{P:GroupFusionBracket} are shown in Section~\ref{S:FusionSystems}.

\smallskip

Throughout this paper, homomorphisms are written on the right hand side of the argument (and the composition of homomorphisms is defined accordingly). The reader is referred to
\cite[Sections~I.1-I.7]{Aschbacher/Kessar/Oliver:2011} for basic definitions and results on fusion systems.

\subsection*{Acknowledgements} I'd like to thank Bernd Stellmacher for pointing out to me that Meierfrankenfeld's Lemma (Lemma~\ref{L:Meierfrankenfeld}) holds. Moreover, I'm very grateful to Ulrich Meierfrankenfeld for allowing me to include his lemma and its proof in this paper.

\section{Group-theoretic results}\label{S:Groups}

\textbf{Throughout this section let $G$ be a finite group and $S\in\Syl_p(G)$.}

\subsection{The proof of Meierfrankenfeld's Lemma}\label{SS:Meierfrankenfeld}

The following lemma and the idea for its proof are due to Ulrich Meierfrankenfeld. We use throughout that two subnormal subgroups generate according to Wielandt's Join Theorem a subnormal subgroup. Moreover, we use frequently that the intersection of $S$ with a subnormal subgroup $H$ of $G$ is a Sylow $p$-subgroup of $H$.

\begin{lemma}[Meierfrankenfeld]\label{L:Meierfrankenfeld}
 Let $H_1$ and $H_2$ be subnormal subgroups of $G$. Then
 \[\<H_1,H_2\>\cap S=\<H_1\cap S,H_2\cap S\>.\]
\end{lemma}

\begin{proof}
Let $(G,H_1,H_2)$ be a counterexample such that first $|G|$ and then $|G:H_1|+|G:H_2|$ is minimal. The minimality of $G$ implies that $G=\<H_1,H_2\>$. As $(G,H_1,H_2)$ is a counterexample, we have moreover  $H_1\not\leq H_2$ and $H_2\not\leq H_1$. Similarly, one observes that
\begin{equation}\label{E:TrivialCases0}
G\neq S\mbox{ and }p\mbox{ divides the order of }G.
\end{equation}
Assume $H_1\unlhd G$. Then $G=\<H_1,H_2\>=H_1H_2$, $H_1\cap S\unlhd S$ and $\<H_1\cap S,H_2\cap S\>=(H_1\cap S)(H_2\cap S)$. As $S\cap H_1\cap H_2\in\Syl_p(H_1\cap H_2)$, the order formula for products of subgroups (cf. \cite[1.1.6]{KS}) yields that $(H_1\cap S)(H_2\cap S)$ is a Sylow $p$-subgroup of $G=H_1H_2$ and thus equal to $S$. This contradicts the assumption that $(G,H_1,H_2)$ is a counterexample and shows thus that
\begin{equation}\label{E:H1NotNormal}
 H_1\mbox{ is not normal in }G.
\end{equation}
As $H_i\neq G$ is subnormal in $G$ for each $i=1,2$, there exists $M_i\unlhd G$ such that $H_i\leq M_i$ and $G/M_i$ is simple.
Set
\[H:=\<H_1,H_2\cap M_1\>\]
Since $H_1$ and $H_2\cap M_1$ are subnormal in $G$, it follows from Wielandt's Join Theorem that $H$ is subnormal in $G$ and thus $ H\cap S\in\Syl_p(H)$. As $H\leq M_1<G$, the minimality of $|G|$ implies that $(H,H_1,H_2\cap M_1)$ is not a counterexample. Hence,
\begin{equation}\label{E:HcapS}
H\cap S=\<H_1\cap S,H_2\cap M_1\cap S\>.
\end{equation}
Assume now that $H_2\cap M_1\not\leq H_1$ and so $|G:H|<|G:H_1|$. As $G=\<H,H_2\>$, the minimality of $|G:H_1|+|G:H_2|$ implies then that $S=\<H\cap S,H_2\cap S\>$. So $S=\<H_1\cap S,H_2\cap S\>$ by \eqref{E:HcapS},
contradicting the assumption that $(G,H_1,H_2)$ is a counterexample. This proves that $H_2\cap M_1\leq H_1$, which yields $H_1\cap H_2=H_2\cap M_1\unlhd H_2$. A symmetric argument shows that $H_1\cap H_2=H_1\cap M_2\unlhd H_1$ and thus
\[H_1\cap H_2=H_1\cap M_2=H_2\cap M_1\unlhd \<H_1,H_2\>=G.\]
Consider now $\ov{G}:=G/H_1\cap H_2$. If $H_1\cap H_2\neq 1$, then the minimality of $|G|$ implies that $(\ov{G},\ov{H_1},\ov{H_2})$ is not a counterexample and hence 
\[\ov{S}=\<\ov{H_1}\cap \ov{S},\ov{H_2}\cap\ov{S}\>=\<\ov{H_1\cap S},\ov{H_2\cap S}\>.\]
This implies however that $S=\<H_1\cap S,H_2\cap S\>$, again a contradiction. Thus, we have
\[1=H_1\cap H_2=H_1\cap M_2=H_2\cap M_1.\]
Let $i\in\{1,2\}$ and set $j:=3-i$. Recall that $G/M_j$ is simple and has thus no non-trivial subnormal subgroup. Since $H_i\not\leq M_j$ and $H_i\cap M_j=1$, it follows that
\[G/M_j=H_iM_j/M_j\cong H_i/H_i\cap M_j\cong H_i.\]
In particular, $H_i$ is simple. As $H_i$ is subnormal, it follows that $H_i$ is a component of $G$ if $H_i$ is non-abelian. Otherwise, $H_i$ is cyclic of prime order $q$, and then contained in $O_q(G)$ (cf. \cite[6.3.1]{KS}). Since any two distinct components centralize each other, and every component centralizes
\[F(G)=\bigtimes\limits_{\tiny{\mbox{$q$ prime}}}O_q(G),\]
it follows that $H_1$ and $H_2$ either centralize each other, or $\<H_1,H_2\>\leq O_q(G)$ for the some prime $q$. As $G=\<H_1,H_2\>$, we get a contradiction to \eqref{E:H1NotNormal} in the first case, and a contradiction to \eqref{E:TrivialCases0} in the second case. This proves the assertion.
\end{proof}

\subsection{The proof of Theorem~\ref{T:Groups}}\label{SS:GroupFusion}

As before we assume that $G$ is a finite group and $S\in\Syl_p(G)$. If $\Delta$ is the set of all subgroups of $S$, then $(G,\Delta,S)$ is a locality. Thus, we can adopt the usual notation for localities and we could cite theorems on localities. However, since Theorem~\ref{T:Groups} might be of independent interest to group theorists, we will keep its proof self-contained. Adopting the usual notation for localities, we define for $f\in G$
\[S_f:=\{s\in S\colon s^f\in S\}=(S\cap S^f)^{f^{-1}}.\]
Note that $S_f$ is a subgroup of $S$. If $u=(f_1,\dots,f_n)$ is a word in $G$, then we set similarly
\[S_u:=\{s\in S\colon s^{f_1\cdots f_i}\in S\mbox{ for all }i=1,2,\dots,n\}.\]
Again one observes easily that $S_u$ is a subgroup of $S$. Notice moreover that the following holds:
\begin{equation}\label{E:SuSf}
\mbox{If }u=(f_1,\dots,f_n)\mbox{ is a word in $G$ and }f=f_1f_2\cdots f_n,\mbox{ then }S_u\leq S_f.
\end{equation}

\begin{lemma}\label{L:FrattiniGroupFusion}
Let $N\unlhd G$ and $T:=S\cap N$. Then for every $g\in G$ there exist $n\in N$ and $f\in N_G(T)$ with $g=nf$ and $S_g=S_{(n,f)}$. In particular,
\[\F_S(G)=\<\F_S(NS),\F_S(N_G(T))\>.\]
\end{lemma}

\begin{proof}
It is sufficient to prove the first part of the assertion. Indeed, this part is a special case of Stellmacher's Splitting Lemma for localities \cite[Lemma~3.12]{Chermak:2015}, but we give an elementary proof.

\smallskip

Let $g\in G$. By the Frattini Argument, there exist $m\in N$ and $h\in N_G(T)$ with $g=mh$. Set $P:=S_g$. As $(P^m)^h=P^g\leq S\leq N_G(T)$ and $h\in N_G(T)$, we have $P^m\leq N_G(T)$. Moreover, $P^m\leq SN$, so by a Dedekind Argument $P^m\leq N_G(T)\cap SN=SN_N(T)$. By Sylow's theorem, there exists thus $x\in N_N(T)$ with $P^{mx}=(P^m)^x\leq S$. Notice that $n:=mx\in N$, $f:=x^{-1}h\in N_G(T)$ and $g=mh=nf$. In particular, $S_{(n,f)}\leq S_{nf}=P$ by \eqref{E:SuSf}. As $P^n\leq S$ and $P^{nf}=P^g\leq S$, we have also $P\leq S_{(n,f)}$. This proves the assertion.
\end{proof}

\begin{lemma}\label{L:ProductNormal}
 Let $M,N\unlhd G$ and $g\in MN$. Then there exist $m\in M$ and $n\in N$ with $g=mn$ and $S_g=S_{(m,n)}$. 
\end{lemma}

\begin{proof}
This is a special case of \cite[Theorem~1]{Henke:2015a}, but again, we give an elementary proof. Let $g\in MN$ and write $g=m'n'$ for some $m'\in M$ and $n'\in N$. Notice that $Q:=S_g^{m'}\leq S^{m'}\leq \<S,M\>=SM$. Moreover, $Q^{n'}=S_g^g\leq S$ and thus $Q\leq S^{(n')^{-1}}\leq \<S,N\>=SN$. Hence, $Q\leq SM\cap SN$. Notice that $O^p(SM\cap SN)\leq O^p(SM)\cap O^p(SN)\leq M\cap N$ and thus, as $S$ is a Sylow $p$-subgroup of $SM\cap SN$, we have $SM\cap SN=SO^p(SM\cap SN)=S(M\cap N)$. It follows now from Sylow's theorem that there exists $x\in M\cap N$ with $S_g^{m'x}=Q^x\leq S$. Now $m:=m'x\in M$, $n:=x^{-1}n'\in N$, $g=mn$, $S_g^m\leq S$ and $S_g^{mn}=S_g^g\leq S$. In particular, $S_g\leq S_{(m,n)}$. Together with \eqref{E:SuSf}, we obtain $S_g=S_{(m,n)}$.
\end{proof}

\begin{lemma}\label{L:GetHunlhdG}
Let $H\subn G$ and $M:=\<H^G\>$. If $G=MN_G(H)$, then $H=M\unlhd G$.
\end{lemma}

\begin{proof}
Assume $G=MN_G(H)$ and consider $K:=\<H^M\>$. Notice that $G$ acts on the set of normal subgroups of $M$, and thus $N_G(H)$ acts on the set of normal subgroups of $M$ containing $H$. As $K$ is the smallest normal subgroup of $M$ containing $H$, it follows that $N_G(H)\leq N_G(K)$. By construction, $K\unlhd M$, and so $K$ is normal in $G=MN_G(H)$. This implies that $M=\<H^G\>=K$. Using the definition of $K$ and the fact that $H$ is subnormal in $M$, we can therefore conclude  that $H=M\unlhd G$.
\end{proof}

\begin{lemma}\label{L:ProductNormalSubnormal}
Let $N\unlhd G$ and $H\subn G$ such that $S$ normalizes $H$. Then for every $g\in NH$, there exist $n\in N$ and $h\in H$ with $g=nh$ and $S_g=S_{(n,h)}$. In particular,
\[\F_S(NHS)=\<\F_S(NS),\F_S(HS)\>.\]
\end{lemma}

\begin{proof}
It is sufficient to show that, for every $g\in NH$, there exist $n\in N$ and $h\in H$ with $g=nh$ and $S_g=S_{(n,h)}$.
Let $(G,N,H,g)$ be a counterexample to that assertion such that $|G|+|N|$ is minimal. Set $M:=\<H^G\>$. 

\smallskip

Notice that $g\in NH\subseteq NM$. Hence, by Lemma~\ref{L:ProductNormal}, there exist $n'\in N$ and $m\in M$ such that $g=n'm$ and $S_g=S_{(n',m)}$. In particular, as $(G,N,H,g)$ is a counterexample, it follows
\begin{equation}\label{E:HneqM}
H\neq M.
\end{equation}
Write $g=ab$ with $a\in N$ and $b\in H$. Then $ab=n'm$ and so $mb^{-1}=(n')^{-1}a\in M\cap N$. We obtain thus $m\in (M\cap N)b\subseteq (M\cap N)H$.

\smallskip

Assume now that $N\not\leq M$. Then $M\cap N$ is a proper subgroup of $N$, which is normal in $G$. The minimality of $|G|+|N|$ yields thus that $(G,M\cap N,H,m)$ is not a counterexample. Hence, there exist $y\in M\cap N$ and $h\in H$ with $m=yh$ and $S_m=S_{(y,h)}$. This implies $g=n'm=n'yh=(n'y)h$ and
\[S_g=S_{(n',m)}=S_{(n',y,h)}\leq S_{(n'y,h)}\leq S_{(n'y)h}=S_g,\]
where the second equality uses $S_m=S_{(y,h)}$ and the inclusions use \eqref{E:SuSf}.
So $n:=n'y\in N$ with $g=nh$ and $S_{(n,h)}=S_g$. This contradicts the assumption that $(G,N,H,g)$ is a counterexample. Hence, we have shown that
\[N\leq M\leq MS.\]
By \eqref{E:HneqM} and Lemma~\ref{L:GetHunlhdG}, we have $G\neq MN_G(H)$. In particular, as $S\leq N_G(H)$, it follows
$MS\neq G$. Thus, the minimality of $|G|+|N|$ yields that $(MS,N,H,g)$ is not a counterexample. However, this implies that $(G,N,H,g)$ is not a counterexample, contradicting our assumption.
\end{proof}

\begin{lemma}\label{L:SubnormalSeries}
Let $H\subn G$. Then there exists a subnormal series
\[H=H_0\unlhd H_1\unlhd \cdots \unlhd H_{n-1}\unlhd H_n=G\]
such that $H_{i-1}=\<H^{H_i}\>$ for $i=1,2,\dots,n$. Such a subnormal series is always invariant under conjugation by $N_G(H)$.
\end{lemma}

\begin{proof}
If $H$ is properly contained in a subgroup $H_i$ of $G$, then $\<H^{H_i}\>$ is a proper subgroup of $H_i$, as $H$ is subnormal in $H_i$. It follows thus by induction on $|G:H|$ that we can construct a subnormal series as above.

\smallskip

If $i>1$ and $H_i$ is $N_G(H)$-invariant, then $N_G(H)$ acts via conjugation on the set of normal subgroups of $H_i$ containing $H$, and thus also on $H_{i-1}=\<H^{H_i}\>$, which is the smallest normal subgroup of $H_i$ containing $H$. Thus, by induction on $n-i$, one proves that $H_i$ is $N_G(H)$-invariant for $i=1,2,\dots,n$.
\end{proof}

We are now in a position to give the proof of Theorem~\ref{T:Groups}. In the proof we use the following notation: If $X$ and $Y$ are subgroups of $G$, then $\Aut_X(Y)$ is the group of automorphisms of $Y$ which are induced by conjugation with an element of $N_X(Y)$.

\begin{proof}[Proof of Theorem~\ref{T:Groups}]
Let $(G,H_1,H_2)$ be a counterexample to Theorem~\ref{T:Groups} such that first $|G|$ and then $|G:H_1|+|G:H_2|$ is minimal. Set
\[S_1:=S\cap H_1\mbox{ and }S_2:=S\cap H_2.\]
The minimality of $|G|$ together with Meierfrankenfeld's Lemma~\ref{L:Meierfrankenfeld} yields that
\begin{equation}\label{E:Generation}
G=\<H_1,H_2\>\mbox{ and }S=\<S_1,S_2\>.
\end{equation}
If $H_1\neq H_1^x$ for some $x\in S$, then $H_1<H_1^*:=\<H_1,H_1^x\>\subn G$ by Wielandt's Join Theorem. Hence, it follows from the minimality of $|G:H_1|+|G:H_2|$ that
\[\F_S(G)=\<\;\F_S(\<H_1^*,S\>),\;\F_S(\<H_2,S\>)\;\>.\]
As $\<H_1^*,S\>=\<H_1,S\>$, this contradicts the assumption that $(G,H_1,H_2)$ is a counterexample. So $H_1$ is $S$-invariant. As the situation is symmetric in $H_1$ and $H_2$, we have thus shown that
\begin{equation}\label{E:SinNHi}
 S\leq N_G(H_1)\cap N_G(H_2).
\end{equation}
In particular, we need to show that $\F_S(G)=\<\F_S(H_1S),\F_S(H_2S)\>$ to obtain a contradiction to the assumption that $(G,H_1,H_2)$ is a counterexample. We will use \eqref{E:Generation} and \eqref{E:SinNHi} throughout this proof, most of the time without further reference.

\smallskip

Set
\[N_i:=\<H_i^G\>\mbox{ for }i=1,2.\]
As $(G,H_1,H_2)$ is a counterexample, it follows from Lemma~\ref{L:ProductNormalSubnormal} (applied with $(H_i,H_{3-i})$ in place of $(N,H)$) that $H_i$ is not normal in $G$ for each $i=1,2$. In particular, Lemma~\ref{L:GetHunlhdG} together with \eqref{E:SinNHi} yields
\begin{equation}\label{E:NiSneqG}
N_1S\neq G\neq N_2S.
\end{equation}
We prove now that
\begin{equation}\label{E:H1capH2normal}
 H_1\cap N_2=H_1\cap H_2=H_2\cap N_1\unlhd G.
\end{equation}
For the proof assume first that $H_1\cap N_2\not\leq H_2$. Then
\[H_2<\tilde{H}_2:=\<H_2,H_1\cap N_2\>.
\]
Notice that $H_1\cap N_2$ is subnormal in $G$ and hence, by Wielandt's Join theorem, $\tilde{H}_2$ is subnormal in $G$. The minimality of $|G:H_1|+|G:H_2|$ yields thus
\[\F_S(G)=\<\F_S(H_1S),\F_S(\tilde{H}_2S)\>.\]
Notice that $K_1:=H_1\cap (N_2S)\subn N_2S$ and $K_1\cap S=S_1$. Using \eqref{E:Generation} and a Dedekind argument, observe moreover that
\[K_1=H_1\cap (N_2S)=H_1\cap (N_2S_1)=(H_1\cap N_2)S_1.\]
In particular, we can conclude that $\tilde{H}_2S=\<K_1,H_2\>$ and $K_1S=(H_1\cap N_2)S$. By \eqref{E:NiSneqG} and the minimality of $|G|$, $(N_2S,K_1,H_2)$ is not a counterexample. Hence,
\[\F_S(\tilde{H}_2S)=\F_S(\<K_1,H_2\>)=\<\F_S(K_1S),\F_S(H_2S)\>\]
where $\F_S(K_1S)=\F_S((H_1\cap N_2)S)\subseteq \F_S(H_1S)$. Hence,
\[\F_S(G)=\<\F_S(H_1S),\F_S(\tilde{H}_2S)\>=\<\F_S(H_1S),\F_S(H_2S)\>,\]
which contradicts the assumption that $(G,H_1,H_2)$ is a counterexample. This shows that $H_1\cap N_2\leq H_2$ and so $H_1\cap N_2=H_1\cap H_2$. A symmetric argument gives $H_2\cap N_1=H_1\cap H_2$. Thus,
\[H_1\cap N_2=H_1\cap H_2=H_2\cap N_1\unlhd \<H_1,H_2\>=G\]
and so \eqref{E:H1capH2normal} holds. We argue next that
\begin{equation}\label{E:FactorH1}
 H_1=(H_1\cap H_2)N_{H_1}(S_2).
\end{equation}
For the proof recall that $H_1$ is $S$-invariant by \eqref{E:SinNHi} and $S_2\leq N_2\unlhd G$. Hence \eqref{E:H1capH2normal}  yields
\[[H_1,S_2]\leq H_1\cap N_2=H_1\cap H_2\unlhd G.\]
Thus, $H_1$ normalizes $(H_1\cap H_2)S_2$. In particular, $H_1$ acts on $\Syl_p((H_1\cap H_2)S_2)$, and $H_1\cap H_2$ acts transitively on this set as $S_2\in \Syl_p((H_1\cap H_2)S_2)$. The general Frattini Argument \cite[3.1.4]{KS} implies thus \eqref{E:FactorH1}. As a next step we prove
\begin{equation}\label{E:OupperpH2}
O^{p^\prime}(H_2)=O^{p^\prime}(N_2)\unlhd G.
\end{equation}
For the proof, Lemma~\ref{L:SubnormalSeries} allows us to pick a subnormal series
\[H_2=M_0\unlhd M_1\unlhd\cdots\unlhd M_k=N_2\unlhd M_{k+1}=G\]
for $H_2$ in $G$ with $M_i=\<H_2^{M_{i+1}}\>$ for $i=0,1,2,\dots,k$. Moreover, $M_0,M_1,\dots,M_k$ are $N_G(H_2)$-invariant. In particular, $M_0,M_1,\dots,M_k$ are $S$-invariant by \eqref{E:SinNHi}. Fix now $i\in\{0,1,2,\dots,k\}$ minimal with $O^{p^\prime}(M_i)=O^{p^\prime}(N_2)$. Notice that such $i$ exists as $M_k=N_2$. Assume \eqref{E:OupperpH2} fails. Then $i>0$ and $H_2<M_i$. In particular, the minimality of $|G:H_1|+|G:H_2|$ yields that $(G,H_1,M_i)$ is not a counterexample and thus
\[\F_S(G)=\<\F_S(H_1S),\F_S(M_iS)\>.\]
Note that $S\cap N_2$ is strongly closed in $S$ as $N_2\unlhd G$. So the above equality implies
\[\Aut_G(S\cap N_2)=\<\Aut_{H_1}(S\cap N_2),\Aut_{M_i}(S\cap N_2),\Aut_S(S\cap N_2)\>.\]
Recall that $M_{i-1}$ is normalized by $M_i$ and by $S$. Thus, $M_{i-1}\cap S$ is $\Aut_{M_i}(S\cap N_2)$-invariant and $\Aut_S(S\cap N_2)$-invariant. Moreover, using \eqref{E:H1capH2normal} and that fact that $H_1$ is $S$-invariant, we see that
\[[N_{H_1}(S\cap N_2),S\cap M_{i-1}]\leq [N_{H_1}(S\cap N_2),S\cap N_2]\leq S\cap N_2\cap H_1 =S\cap H_1\cap H_2\leq S_2\leq S\cap M_{i-1}.
\]
This implies that $S\cap M_{i-1}$ is $\Aut_{H_1}(S\cap N_2)$-invariant. Altogether, we have seen that $S\cap M_{i-1}$ is $\Aut_G(S\cap N_2)$-invariant and thus
\[N_G(S\cap N_2)\leq N_G(S\cap M_{i-1}).\]
As $O^{p^\prime}(N_2)$ is normal in $G$ and by assumption contained in $M_i$, a Frattini Argument yields now
\[G=O^{p^\prime}(N_2)N_G(S\cap N_2)=M_iN_G(S\cap M_{i-1}).\]
Hence, by a Dedekind Argument, setting $X:=N_{M_{i+1}}(S\cap M_{i-1})$, we have  $M_{i+1}=M_iX$. As $H_2\leq M_{i-1}\unlhd M_i\unlhd M_{i+1}$, it follows that
\begin{eqnarray*}
M_i&=&\<H_2^{M_{i+1}}\>=\<M_{i-1}^{M_{i+1}}\>=\<M_{i-1}^X\>\\
&=& \prod_{x\in X}M_{i-1}^x.
\end{eqnarray*}
Notice that, for every $x\in X$, $M_{i-1}^x$ is a normal subgroup of $M_i$ with $M_{i-1}^x\cap S=(M_{i-1}\cap S)^x=M_{i-1}\cap S$. Hence, using Lemma~\ref{L:Meierfrankenfeld} (or an elementary argument for this special case), one sees that
\[M_i\cap S=\prod_{x\in X}(M_{i-1}^x\cap S)=M_{i-1}\cap S.\]
As $M_{i-1}\unlhd M_i$, it follows that $O^{p^\prime}(M_{i-1})=O^{p^\prime}(M_i)=O^{p^\prime}(N_2)$, contradicting the minimality of $i$. Thus, \eqref{E:OupperpH2} holds. We show next that
\begin{equation}\label{E:S2Normal}
 S_2\unlhd G.
\end{equation}
For the proof we use that, by \eqref{E:H1capH2normal} and \eqref{E:OupperpH2}, $H_1\cap H_2$ and $O^{p^\prime}(H_2)$ are normal in $G$. In particular,
\[\hat{H}_2:=(H_1\cap H_2)O^{p^\prime}(H_2)\unlhd G.\]
A Frattini Argument yields $H_2=O^{p^\prime}(H_2)N_{H_2}(S_2)$ and \eqref{E:FactorH1} gives $H_1=(H_1\cap H_2)N_{H_1}(S_2)$. Hence,
\[G=\<H_1,H_2\>=\hat{H}_2\<N_{H_1}(S_2),N_{H_2}(S_2)\>.\]
Notice that $\hat{H}_2\leq H_2$ and so $N_{\hat{H}_2}(S_2)\leq N_{H_2}(S_2)$. Hence, a Dedekind Argument gives that
\[N_G(S_2)=N_{\hat{H}_2}(S_2)\<N_{H_1}(S_2),N_{H_2}(S_2)\>=\<N_{H_1}(S_2),N_{H_2}(S_2)\>.\]
Observe that $S_2\unlhd S$ by \eqref{E:SinNHi}. Assume now that $S_2$ is not normal in $G$. Then the minimality of $|G|$ yields that $(N_G(S_2),N_{H_1}(S_2),N_{H_2}(S_2))$ is not a counterexample, i.e.
\[\F_S(N_G(S_2))=\<\;\F_S(N_{H_1}(S_2)S),\;\F_S(N_{H_2}(S_2)S)\;\>.\]
As $O^{p^\prime}(H_2)$ is normal in $G$, Lemma~\ref{L:FrattiniGroupFusion} yields thus that
\[\F_S(G)=\<\F_S(O^{p^\prime}(H_2)S),\F_S(N_G(S_2))\>=\<\F_S(H_1S),\F_S(H_2S)\>,\]
which contradicts the assumption that $(G,H_1,H_2)$ is a counterexample. Hence, \eqref{E:S2Normal} holds.

\smallskip

The fact that $S_2$ is normal in $G$ implies by Wielandt's Join Theorem in particular that $H_1S_2$ is subnormal in $G$. If $S_2\not\leq H_1$, then the minimality of $|G:H_1|+|G:H_2|$ yields that $(G,H_1S_2,H_2)$ is not a counterexample. As $(H_1S_2)S=H_1S$, this would imply that $(G,H_1,H_2)$ is not a counterexample, a contradiction. Hence, $S_2\leq H_1\cap S=S_1$. A symmetric argument yields $S_1\leq S_2$ and thus
\[S=S_1=S_2\unlhd G.\]
It follows now easily that $(G,H_1,H_2)$ is not a counterexample contradicting our assumption.
\end{proof}

\section{Some background on partial groups and localities}\label{S:PartialLocalities}

\subsection{Partial groups and localities}

We will start by summarizing some basic background on partial groups and localities here, but the reader is referred to Chermak's original papers \cite{Chermak:2013,Chermak:2015} or to the summary in  \cite[Section~3]{Chermak/Henke} for a detailed introduction to the required definitions and results concerning partial groups and localities.

\smallskip

Following Chermak's notation, we write $\W(\L)$ for the set of words in a set $\L$, $\emptyset$ for the empty word, and $v_1\circ v_2\circ\cdots\circ v_n$ for the concatenation of words $v_1,\dots,v_n\in\W(\L)$. Recall that a partial group consists of a set $\L$, a ``product'' $\Pi\colon \D\rightarrow \L$ defined on $\D\subseteq\W(\L)$, and an involutory bijection  $\L\rightarrow \L,f\mapsto f^{-1}$ called an ``inversion map'', subject to certain group-like axioms (cf. \cite[Definition~2.1]{Chermak:2013} or \cite[Definition~1.1]{Chermak:2015}). If $\L$ is a partial group with a product $\Pi\colon \D\rightarrow \L$ then, given $(x_1,\dots,x_n)\in\D$, we write also $x_1x_2\cdots x_n$ for $\Pi(x_1,\dots,x_n)$.  We will moreover use the following definitions.
\begin{itemize}
 \item A subset $\H\subseteq\L$ is called a \emph{partial subgroup} of $\L$ if $\Pi(w)\in\H$ for every $w\in\W(\H)\cap \D$ and $h^{-1}\in\H$ for every $h\in\H$.
\item A partial subgroup $\H$ of $\L$ is called a \emph{subgroup} of $\L$ if $\W(\H)\subseteq\D$. Observe that every subgroup of $\L$ forms an actual group. We call a subgroup $\H$ of $\L$ a \emph{$p$-subgroup} if it is a $p$-group.
\item Given $f\in\L$, we write $\D(f):=\{x\in\L\colon (f^{-1},x,f)\in\D\}$ for the set of elements $x\in\L$ for which the \emph{conjugate} $x^f:=\Pi(f^{-1},x,f)$ is defined. This gives us a conjugation map $c_f\colon\D(f)\rightarrow \L$ defined by $x\mapsto x^f$.
\item Let $S$ be a $p$-subgroup of $\L$. Then set
\[S_f:=\{x\in S\colon x\in\D(f),\;x^f\in S\}\mbox{ for all }f\in\L.\]
More generally, if $w=(f_1,\dots,f_n)\in\W(\L)$, then write $S_w$ for the subset of $S$ consisting of all elements $s\in S$ such that there exists a series $s=s_0,s_1,\dots,s_n\in S$ with $s_{i-1}^{f_i}=s_i$ for $i=1,2,\dots,n$.
\item We say that a partial subgroup $\N$ of $\L$ is a \emph{partial normal subgroup} of $\L$ (and write $\N\unlhd\L$) if $x^f\in\N$ for every $f\in\L$ and every $x\in \D(f)\cap \N$.
\item We call a partial subgroup $\H$ of $\L$ a \emph{partial subnormal subgroup} of $\L$ (and write $\H\subn\L$) if there is a sequence $\H_0,\H_1,\dots,\H_k$ of partial subgroups of $\L$ such that
\[\H=\H_0\unlhd\H_1\unlhd\cdots\unlhd \H_{k-1}\unlhd \H_k=\L.\]
\item If $f\in\L$ and $\X\subseteq \D(f)$ set $\X^f:=\{x^f\colon x\in\X\}$. For every $\X\subseteq\L$, call
\[N_\L(\X):=\{f\in\L\colon \X\subseteq\D(f)\mbox{ and }\X^f=\X\}\]
the \emph{normalizer} of $\X$ in $\L$.
\item For $\X\subseteq \L$, we call
\[C_\L(\X):=\{f\in\L\colon \X\subseteq\D(f),\;x^f=x\mbox{ for all }x\in\X\}\]
the \emph{centralizer} of $\X$ in $\L$.
\item For any two subsets $\X,\Y\subseteq\L$, their \emph{product} can be naturally defined by
\[\X\Y:=\{\Pi(x,y)\colon x\in\X,\;y\in\Y,\;(x,y)\in\D\}.\]
\end{itemize}
A \emph{locality} is a triple $(\L,\Delta,S)$ such that $\L$ is a partial group, $S$ is maximal among the $p$-subgroups of $\L$, and $\Delta$ is a set of subgroups of $S$ subject to certain axioms, which imply in particular that $N_\L(P)$ is a subgroup of $\L$ for every $P\in\Delta$. As part of the axioms, a word $w=(f_1,\dots,f_n)$ is an element of $\D$ if and only if there exist $P_0,P_1,\dots,P_n\in\Delta$ such that
\[P_{i-1}\subseteq\D(f_i)\mbox{ and }P_{i-1}^{f_i}=P_i\mbox{ for } i=1,2,\dots,n.\]
If the above holds, then we say also that $w\in\D$ via $P_0,P_1,\dots,P_n$ or that $w\in\D$ via $P_0$.

\smallskip

If $(\L,\Delta,S)$ is a locality, then \cite[Corollary~2.6]{Chermak:2015} gives that, for every word $w\in\W(\L)$, the subset $S_w$ is a subgroup of $S$ and
\begin{equation}\label{E:Sw}
 S_w\in\Delta\mbox{ if and only if }w\in\D.
\end{equation}
In particular, $S_f$ is a subgroup of $S$ with $S_f\in\Delta$, for every $f\in\L$. We will also frequently use the following property, which follows from  \cite[Lemma~2.3(c)]{Chermak:2015}:
\begin{equation}\label{E:2}
 S_w\leq S_{\Pi(w)}\mbox{ and }(\cdots(S_w^{f_1})^{f_2}\cdots)^{f_n}=S_w^{\Pi(w)}\mbox{ for every }w=(f_1,\dots,f_n)\in\D.
\end{equation}

\smallskip

Fix now a locality $(\L,\Delta,S)$. For $f\in\L$, the conjugation map $c_f\colon S_f\rightarrow S,x\mapsto x^f$ is by \cite[Proposition~2.5(b)]{Chermak:2015} an injective group homomorphism. For every partial subgroup $\H$ of $\L$, we write $\F_{S\cap\H}(\H)$ for the fusion system over $S\cap \H$ which is generated by the conjugation maps $c_h\colon S_h\cap \H\rightarrow S\cap \H,x\mapsto x^h$ with $h\in\H$. In particular, $\F_S(\L)$ is the fusion system over $S$ which is generated by the conjugation maps $c_f\colon S_f\rightarrow S$. We say that $(\L,\Delta,S)$ is a locality \emph{over} $\F$ if $\F=\F_S(\L)$.

\begin{lemma}\label{L:SplitFSL}
Let $(\L,\Delta,S)$ be a locality and $\N\unlhd\L$. Set $T:=S\cap\N$. Then the following hold:
\begin{itemize}
 \item [(a)] For every $g\in\L$, there exist $n\in\N$ and $h\in N_\L(T)$ such that $(n,h)\in\D$, $g=nh$ and $S_g=S_{(n,h)}$.
 \item [(b)] $\F_S(\L)=\<\F_S(\N S),\F_S(N_\L(T))\>$.
\end{itemize}
\end{lemma}

\begin{proof}
\textbf{(a)} Let $g\in\L$. By the Frattini Lemma \cite[Corollary~3.11]{Chermak:2015}, there exist $n\in\N$ and $h\in\L$ such that $(n,h)\in\D$ and $h$ is $\uparrow$-maximal with respect to $\N$ in the sense of \cite[Definition~3.6]{Chermak:2015}. Then $S_g=S_{(n,h)}$ by the Splitting Lemma \cite[Lemma~3.12]{Chermak:2015}. Using first \cite[Proposition~3.9]{Chermak:2015} and then \cite[Lemma~3.1(a)]{Chermak:2015}, it follows that $h\in N_\L(T)$.

\smallskip

\textbf{(b)} As $\F_S(\L)$ is generated by maps of the form $c_g\colon S_g\rightarrow S$, it is sufficient to prove that such a map is in $\<\F_S(\N S),\F_S(N_\L(T))\>$. Fixing $g\in \L$, it follows from (a) that there exist $n\in \N$ and $h\in N_\L(T)$ such that $(n,h)\in\D$, $g=nh$ and $S_g=S_{(n,h)}$. So $c_g\colon  S_g\rightarrow S$ can be written as a composite of restrictions of the conjugation map $c_n\colon S_n\rightarrow S$ (which is a morphism in $\F_S(\N S)$) and of $c_h\colon S_h\rightarrow S$ (which is a morphism in $\F_S(N_\L(T))$). This implies (b).
\end{proof}

\subsection{Linking localities and regular localities}

A finite group $G$ is said to be of \emph{characteristic $p$}, if $C_G(O_p(G))\leq O_p(G)$, where $O_p(G)$ denotes the largest normal $p$-subgroup of $G$. A locality $(\L,\Delta,S)$ is called a \emph{linking locality}, if $\F_S(\L)$ is saturated, $\F_S(\L)^{cr}\subseteq\Delta$ and $N_\L(P)$ is a group of characteristic $p$ for every $P\in\Delta$.

\smallskip

For every fusion system $\F$ over $S$, there is    the set $\F^s$ of \emph{$\F$-subcentric} subgroups of $S$ defined in \cite[Definition~1]{Henke:2015}. It is shown in \cite[Theorem~A]{Henke:2015} that, for every saturated fusion system $\F$, there exists a linking locality $(\L,\Delta,S)$ over $\F$ with $\Delta=\F^s$. We call such a linking locality a \emph{subcentric locality} over $\F$.

\smallskip

Regular localities were first introduced by Chermak \cite{ChermakIII}, but we will refer to the treatment of the subject in \cite{Henke:Regular}. Building on Chermak's work, we introduced in \cite[Definition~9.17]{Henke:Regular} a certain partial normal subgroup $F^*(\L)$ of $\L$, for every linking locality $(\L,\Delta,S)$. Then for any saturated fusion system $\F$ over $S$ a set $\delta(\F)$ of subgroups of $S$ was introduced such that, for any linking locality $(\L,\Delta,S)$ over $\F$, we have
\begin{equation}\label{E:deltaF}
\delta(\F)=\{P\leq S\colon P\cap F^*(\L)\in\F^s\}.
\end{equation}
Indeed, the set $\delta(\F)$ is defined in \cite[Definition~10.1]{Henke:Regular} by the equation \eqref{E:deltaF} if $(\L,\Delta,S)$ is a fixed subcentric linking locality over $\F$. Then it is shown in \cite[Lemma~10.2]{Henke:Regular} that the set $\delta(\F)$ depends only on $\F$ and not on the choice of the subcentric locality $(\L,\Delta,S)$, and that \eqref{E:deltaF} holds actually for every linking locality $(\L,\Delta,S)$.

\smallskip

A linking locality $(\L,\Delta,S)$ is called a \emph{regular locality}, if $\Delta=\delta(\F)$. For every saturated fusion system $\F$, there exists a regular locality over $\F$ (cf. \cite[Lemma~10.4]{Henke:Regular}). Note that \eqref{E:deltaF} holds in particular if $(\L,\Delta,S)$ is a regular locality over $\F$. In that case, we have $\Delta=\delta(\F)$, so \eqref{E:deltaF} yields that $P\in\Delta$ if and only if $P\cap F^*(\L)\in\Delta$. Thus, if $(\L,\Delta,S)$ is a regular locality and $T^*:=F^*(\L)\cap S$, then \eqref{E:Sw} yields
\begin{equation}\label{E:SwcapTstar}
 w\in\D\mbox{ if and only if }S_w\cap T^*\in\Delta.
\end{equation}
In particular, $S_f\cap T^*\in\Delta$ for every $f\in\L$. The next theorem states one of the most important properties of regular localities.

\begin{theorem}[{\cite[Corollary~7.9]{ChermakIII}, \cite[Corollary~10.19]{Henke:Regular}}]\label{T:RegularSubnormal}
Let $(\L,\Delta,S)$ be a regular locality and $\H\subn\L$. Then $\F_{S\cap\H}(\H)$ is saturated and $(\H,\delta(\F_{S\cap\H}(\H)),S\cap\H)$ is a regular locality.
\end{theorem}

The theorem above leads to a natural definition of components of regular localities (cf. \cite[Definition~7.9, Definition~11.1]{Henke:Regular}). Let $(\L,\Delta,S)$ be a regular locality. We will write $\Comp(\L)$ for the set of components of $\L$. If $\K_1,\dots,\K_r\in\Comp(\L)$, then the product $\prod_{i=1}^r\K_i$ does not depend on the order of the factors and is a partial normal subgroup of $F^*(\L)$ (cf. \cite[Proposition~11.7]{Henke:Regular}). In particular, for $\fC\subseteq\Comp(\L)$, the product $\prod_{\K\in\fC}\K$ is well-defined.

\smallskip

The product of all components of $\L$ is denoted by $E(\L)$ and turns out to be a partial normal subgroup of $\L$ (cf. \cite[Lemma~11.13]{Henke:Regular}). We have moreover  $F^*(\L)=E(\L)O_p(\L)$ (cf. \cite[Lemma~11.9]{Henke:Regular}). Note that  Theorem~\ref{T:RegularSubnormal} makes it possible to consider $\Comp(\H)$ and $E(\H)$ for every partial subnormal subgroup $\H$ of $\L$. By \cite[Remark~11.2]{Henke:Regular}, $\Comp(\H)\subseteq\Comp(\L)$. In particular, $E(\H)\subseteq E(\L)\subseteq F^*(\L)$.

\subsection{Some further properties of regular localities} We state now some slightly more specialized results on regular localities which will be needed in the proof of Theorem~\ref{T:RegularLocalities}.

\smallskip

\textbf{Throughout this subsection let $(\L,\Delta,S)$ be a regular locality and $T^*:=S\cap F^*(\L)$.}

\smallskip

By \cite[Corollary~10.5]{Henke:Regular}, we have $T^*\in\delta(\F)=\Delta$. In particular, $N_\L(T^*)$ is a group of characteristic $p$.

\begin{lemma}\label{L:1}
For every $f\in N_\L(T^*)$ and every $g\in\L$, the words $(f,g)$, $(g,f)$, $(f,f^{-1},g,f)$ are in $\D$ and $gf=fg^f$. Moreover,
\[S_{(f,g)}\cap T^*=S_{fg}\cap T^*\mbox{ and }S_{(g,f)}\cap T^*=S_{gf}\cap T^*.\]
\end{lemma}

\begin{proof}
It is a special case of \eqref{E:SwcapTstar} that $S_g\cap T^*\in\Delta$. Hence, as $T^*$ is strongly closed, $(g,f)$, $u:=(f,f^{-1},g,f)$ and $(g,f,f^{-1})$ are in $\D$ via $S_g\cap T^*$. In particular, by the axioms of a partial group, $gf=\Pi(u)=fg^f$ and $(gf)f^{-1}=\Pi(g,f,f^{-1})=g$. So
\[S_{(g,f)}\cap T^*\leq S_{gf}\cap T^*\leq S_{(gf,f^{-1},f)}\cap T^*\leq S_{((gf)f^{-1},f)}\cap T^*=S_{(g,f)}\cap T^*,\]
where the first and the third inclusion use \eqref{E:2} and the second inclusion uses $f\in N_\L(T^*)$. Thus, $S_{(g,f)}\cap T^*=S_{gf}\cap T^*$. Observe also that $(f^{-1},f,g)\in\D$ via $S_g\cap T^*$. In particular, $(f,g)\in\D$ and $g=\Pi(f^{-1},f,g)=f^{-1}(fg)$. It follows that
\[S_{(f,g)}\cap T^*\leq S_{fg}\cap T^*\leq S_{(f,f^{-1},fg)}\cap T^*\leq S_{(f,f^{-1}(fg))}\cap T^*= S_{(f,g)}\cap T^*,\]
where again the first and the third inclusion use \eqref{E:2} and the second inclusion uses $f\in N_\L(T^*)$. Hence, $S_{(f,g)}\cap T^*=S_{fg}\cap T^*$.
\end{proof}

\begin{lemma}\label{L:NLT*}
We have $N_\L(T^*)=N_\L(S\cap E(\L))$.
\end{lemma}

\begin{proof}
By \cite[Lemma~11.9]{Henke:Regular}, we have
\[T^*=S\cap F^*(\L)=(S\cap E(\L))O_p(\L).\]
As $E(\L)\unlhd \L$, $S\cap E(\L)=T^*\cap E(\L)$ and $\L=N_\L(O_p(\L))$, the assertion follows.
\end{proof}

\begin{lemma}\label{L:NLT*actsComp}
$N_\L(T^*)$ acts on $\L$ and also on the set $\Comp(\L)$ of components of $\L$ via conjugation. More precisely, for every $f\in\L$, we have $\D(f)=\L$ and $c_f$ is an automorphism of $\L$.
\end{lemma}

\begin{proof}
It is a consequence of Theorem~2 and Lemma~10.11(c) in \cite{Henke:Regular} that, for every $f\in N_\L(T^*)$, we have $\L=\D(f)$ and the conjugation map $c_f$ is an automorphism of $\L$. Moreover, by \cite[Lemma~2.13]{Henke:NormalizerSubnormal}, $N_\L(T^*)$ acts on the set $\L$ by conjugation. Thus, by \cite[Lemma~11.12]{Henke:Regular}, $N_\L(T^*)$ acts also on the set of components of $\L$.
\end{proof}

\begin{lemma}\label{L:ComponentProducts}
Let $\fC_1,\fC_2\subseteq\Comp(\L)$, $\fC:=\fC_1\cup \fC_2$, $\N_i:=\prod_{\K\in\fC_i}\K$ for $i=1,2$ and $\N:=\prod_{\K\in\fC}\K$. Then the following hold:
\begin{itemize}
 \item [(a)] $\N=\N_1\N_2$ and $\N\cap S=(\N_1\cap S)(\N_2\cap S)$. In particular, if $\fC=\Comp(\L)$, then $E(\L)=\N_1\N_2$ and $E(\L)\cap S=(\N_1\cap S)(\N_2\cap S)$.
 \item [(b)] Suppose $\fC_1\cap \fC_2=\emptyset$. Then $\N_i\subseteq C_\L(\N_{3-i})$ for $i=1,2$. Moreover, for all $n\in\N_1$ and $m\in\N_2$, we have $(n,m)\in\D$ and $S_{nm}\cap T^*=S_{(n,m)}\cap T^*$.
 \item [(c)] $N_\N(T^*)=N_{\N_1}(T^*)N_{\N_2}(T^*)$.
\end{itemize}
\end{lemma}

\begin{proof}
By \cite[Proposition~11.7]{Henke:Regular}, $\N_1$, $\N_2$ and $\N$ are well-defined and partial normal subgroups of $F^*(\L)$. As $F^*(\L)$ forms by  Theorem~\ref{T:RegularSubnormal} a regular locality with Sylow subgroup $T^*$, it follows in particular from \cite[Theorem~1]{Henke:2015} applied with $F^*(\L)$ in place of $\L$ that $(\N_1\N_2)\cap S=(\N_1\N_2)\cap T^*=(\N_1\cap T^*)(\N_2\cap T^*)=(\N_1\cap S)(\N_2\cap S)$. Thus, for (a), it remains only to prove that $\N=\N_1\N_2$. In fact, as $\N_i\subseteq\N$ for $i=1,2$ and $\N$ is a partial subgroup, we have $\N_1\N_2\subseteq \N$. So for (a) it is sufficient to prove that
\begin{equation}\label{E:NinN1N2}
\N\subseteq \N_1\N_2.
\end{equation}
As $N_\L(T^*)$ is a subgroup, we have also $N_{\N_1}(T^*) N_{\N_2}(T^*)\subseteq N_\N(T^*)$. So for (c) it remains only to show that
\begin{equation}\label{E:NNTstarFactorize}
N_\N(T^*)\subseteq N_{\N_1}(T^*) N_{\N_2}(T^*).
\end{equation}
Observe now that, replacing $\fC_2$ by $\fC_2\backslash \fC_1$, we may assume for the proof of \eqref{E:NinN1N2} and \eqref{E:NNTstarFactorize} that $\fC_1\cap \fC_2=\emptyset$. So we will assume this property from now on throughout.

\smallskip

Applying first  \cite[Theorem~11.18(a)]{Henke:Regular} and then \cite[Lemma~4.5, Lemma~4.8]{Henke:Regular}, one sees that $\N=\N_1\N_2$, and that $\N_i\subseteq C_\L(\N_{3-i})$ for $i=1,2$. In particular, \eqref{E:NinN1N2} holds and thus (a).

\smallskip

For the proof of the remaining statement in (b) let now $n\in\N_1$ and $m\in\N_2$. As $\N_1\subseteq C_\L(\N_2)$, it follows from  \cite[Lemma~3.5]{Henke:Regular} that $(n,m)$ and $(m,n)$ are in $\D$ and that $nm=mn$.
In particular, $\N_1$ and $\N_2$ commute in the sense of \cite[Definition~2]{Henke:Regular}, i.e. they are commuting partial normal subgroups of $F^*(\L)$. Hence, it follows from \cite[Theorem~1(d)]{Henke:Regular} applied with $F^*(\L)$ in place of $\L$ that  $S_{nm}\cap T^*=S_{(n,m)}\cap T^*$. This proves (b).

\smallskip

For the proof of \eqref{E:NNTstarFactorize} let $f\in N_\N(T^*)$. By (a), we may pick $n\in\N_1$ and $m\in\N_2$ with $(n,m)\in\D$ and $f=nm$. It follows then from (b) that  $T^*=T^*\cap S_f=T^*\cap S_{nm}\leq S_{(n,m)}$. As $T^*$ is strongly closed in $\F_S(\L)$, it follows that $n\in N_{\N_1}(T^*)$ and $m\in N_{\N_2}(T^*)$. This proves \eqref{E:NNTstarFactorize} and thus (c).
\end{proof}

\begin{lemma}\label{L:NHT*}
Let $\H\subn \L$. Then $N_\H(T^*)=N_\H(E(\H)\cap S)\subn N_\L(T^*)$ and $S\cap \H=S\cap N_\H(T^*)$.
\end{lemma}

\begin{proof}
Since $T^*\unlhd S$, we have $S\cap \H\subseteq N_\H(T^*)$ and thus $S\cap \H=S\cap N_\H(T^*)$. Since $\H\subn\L$, it follows moreover easily that $N_\H(T^*)=\H\cap N_\L(T^*)$ is subnormal in $N_\L(T^*)$.

\smallskip

As mentioned before, it follows from \cite[Remark~11.2]{Henke:Regular} that $E(\H)\subseteq E(\L)\subseteq F^*(\L)$. Thus, $E(\H)\cap S=E(\H)\cap T^*$. Since $E(\H)\unlhd \H$ by \cite[Lemma~11.13]{Henke:Regular}, it follows that $N_\H(T^*)\subseteq N_\H(T^*\cap E(\H))=N_\H(S\cap E(\H))$. It remains thus only to show that $N_\H(S\cap E(\H))\subseteq N_\H(T^*)$.

\smallskip

Applying first \cite[Theorem~11.18(c)]{Henke:Regular} and then \cite[Lemma~4.8]{Henke:Regular}, one sees that $\H$ is in the centralizer of $\M:=\prod_{\K\in\Comp(\L)\backslash\Comp(\H)}\K$. In particular, $\H\subseteq C_\L(\M\cap S)$. Moreover, by Lemma~\ref{L:ComponentProducts}(a), $E(\L)\cap S=(E(\H)\cap S)(\M\cap S)$. Using  Lemma~\ref{L:NLT*}, we can thus conclude that $N_{\H}(S\cap E(\H))\subseteq N_\L(S\cap E(\L))=N_\L(T^*)$. This proves the assertion.
\end{proof}




\section{Wielandt's Join Theorem for regular localities}\label{S:Regular}

\subsection{Some products in regular localities} In this subsection we prove  some general results on products in regular localities, which allow us later to reduce the proof of Theorem~\ref{T:RegularLocalities} to Wielandt's Join Theorem for groups and to Meierfrankenfeld's Lemma.

\smallskip

\textbf{Throughout this subsection let $(\L,\Delta,S)$ be a regular locality and $T^*:=S\cap F^*(\L)$.}

\begin{lemma}\label{L:2a}
Let $H\leq N_\L(T^*)$ and $\N\unlhd \L$. Then $\N H=H\N$ is a partial subgroup of $\L$. If $H\cap S\in\Syl_p(H)$, then $(\N H)\cap S=(\N\cap S)(H\cap S)$ is a maximal $p$-subgroup of $\N H$.
\end{lemma}

\begin{proof}
It is shown in \cite[Theorem~6.1(a)]{Grazian/Henke} that $\N H=H\N$ is a partial subgroup; essentially, the argument uses the property stated in Lemma~\ref{L:1}. Suppose now $H\cap S\in\Syl_p(H)$. As $N_\N(T^*)\unlhd N_\L(T^*)$ with $\N\cap S=N_\N(T^*)\cap S\in\Syl_p(N_\N(T^*))$,  it follows that 
\[S_0:=(\N\cap S)(H\cap S)\in\Syl_p(N_\N(T^*)H).\] 
Thus, it is a consequence of \cite[Theorem~6.1(c)]{Grazian/Henke} that $S_0$ is a maximal $p$-subgroup of $\N H$. Since $S_0\leq S\cap (\N H)$ and $S\cap (\N H)$ is a $p$-subgroup of $\N H$, we can conclude that $S\cap \N H=S_0$ is a maximal $p$-subgroup of $\N H$.
\end{proof}

\begin{lemma}\label{L:2b}
Let $H\subn N_\L(T^*)$. Then $\hat{\H}:=E(\L)H=HE(\L)\subn\L$ and $E(\hat{\H})=E(\L)$.
\end{lemma}

\begin{proof}
It is a special case of \cite[Theorem~2]{Henke:NK} that $\hat{\H}$ is subnormal in $\L$, but we supply a shorter direct argument here: Let $H=H_0\unlhd H_1\unlhd\cdots \unlhd H_n=N_\L(T^*)$ be a subnormal series for $H$ in $N_\L(T^*)$. Then by Lemma~\ref{L:2a}, $E(\L) H_i=H_i E(\L)$ is a partial subgroup of $\L$ for all $i=0,1,2,\dots,n$. By the Frattini Lemma \cite[Corollary~3.11]{Chermak:2015} and Lemma~\ref{L:NLT*}, we have $\L=N_\L(S\cap E(\L))E(\L)=N_\L(T^*)E(\L)=H_nE(\L)$. Thus, it is sufficient to prove that $H_{i-1}E(\L)\unlhd H_iE(\L)$ for $i=1,2,\dots,n$. So fix $i\in\{1,2,\dots,n\}$, $x,y\in E(\L)$, $h\in H_{i-1}$ and $f\in H_i$ with \[w:=((fy)^{-1},hx,fy)\in\D.\]
By \cite[Lemma~1.4(f)]{Chermak:2015}, $(fy)^{-1}=y^{-1}f^{-1}$. Moreover, Lemma~\ref{L:1} gives
\[S_{hx}\cap T^*=S_{(h,x)}\cap T^*,\]
\[S_{fy}\cap T^*=S_{(f,y)}\cap T^*\mbox{ and }S_{(fy)^{-1}}\cap T^*=S_{y^{-1}f^{-1}}\cap T^*=S_{(y^{-1},f^{-1})}\cap T^*.\]
Hence, setting $u:=(y^{-1},f^{-1},h,x,f,y)\in\D$, it follows that $S_u\cap T^*=S_w\cap T^*$. As $w\in\D$, the property \eqref{E:SwcapTstar} yields now first that  $S_u\cap T^*=S_w\cap T^*\in\Delta$ and then $u\in\D$. Using $f\in N_\L(T^*)$, we see also that
\[v:=(y^{-1},f^{-1},h,f,f^{-1},x,f,y)\in\D\]
via $S_w\cap T^*$. Hence,
\[(hx)^{fy}=\Pi(w)=\Pi(u)=\Pi(v)=(h^fx^f)^y.\]
As $h\in H_{i-1}\unlhd H_i$ and $f\in H_i$, we have $h^f\in H_{i-1}$. Moreover, $x\in E(\L)\unlhd\L$ implies $x^f\in E(\L)$. Hence, $h^fx^f\in H_{i-1}E(\L)$. Since $y\in E(\L)\subseteq H_{i-1}E(\L)$ and $H_{i-1}E(\L)$ is a partial subgroup, it follows $(hx)^{fy}=(h^fx^f)^y\in H_{i-1}E(\L)$. This proves $H_{i-1}E(\L)\unlhd H_iE(\L)$ and thus
$\hat{\H}:=E(\L) H=H E(\L)\subn \L$. It follows from \cite[Remark~11.2]{Henke:Regular} that $\Comp(\hat{\H})=\Comp(\L)$ and thus $E(\hat{\H})=E(\L)$.
\end{proof}

Recall that, by Lemma~\ref{L:NLT*actsComp}, $N_\L(T^*)$ acts on $\L$ and on the set of components of $\L$. In particular, if $H\leq N_\L(T^*)$, then it makes sense to say that a subset of $\Comp(\L)$ is $H$-invariant.


\begin{lemma}\label{L:3}
 Let $H\subn N_\L(T^*)$ and $\fC\subseteq\Comp(\L)$ be $H$-invariant. Set \[\N:=\prod_{\K\in\fC}\K\mbox{ and }\M:=\prod_{\K\in\Comp(\L)\backslash\fC}\K.\] Assume $H\subseteq C_\L(\M)$. Then $\N H\unlhd E(\L)H$, and $\N H=H\N$ is a partial subnormal subgroup of $\L$ with $\N H\cap S=(\N\cap S)(H\cap S)$. Moreover,
 \[N_{\N H}(S\cap \N)=N_\N(S\cap \N)H=N_\N(T^*)H,\;\Comp(\N H)=\fC\mbox{ and }E(\N H)=\N\unlhd \N H.\]
\end{lemma}

\begin{proof}
By Lemma~\ref{L:2b}, $\hat{\H}:=E(\L)H=HE(\L)$ is a partial subnormal subgroup of $\L$ with $E(\L)=E(\hat{\H})$. We argue now that
\begin{equation}\label{E:NunlhdFstarH}
\N\unlhd \hat{\H}.
\end{equation}
For that let $x\in E(\L)$, $h\in H$ and $n\in \N$ with $u:=((xh)^{-1},n,(xh))\in\D$. By \cite[Lemma~1.4(f)]{Chermak:2015}, $(xh)^{-1}=h^{-1}x^{-1}$. Set $v:=(h^{-1},x^{-1},n,x,h)$. It follows from Lemma~\ref{L:1} that $S_v\cap T^*=S_u\cap T^*$. Now \eqref{E:SwcapTstar} yields first $S_v\cap T^*=S_u\cap T^*\in\Delta$ and then $v\in\D$. Hence,
\[n^{xh}=\Pi(u)=\Pi(v)=(n^x)^h.\]
By \cite[Proposition~11.7]{Henke:Regular}, $\N\unlhd F^*(\L)\supseteq E(\L)$ and thus $n^x\in\N$. As $\fC$ is by assumption $H$-invariant and since $H$ induces automorphisms of $\L$ via conjugation (cf. Lemma~\ref{L:NLT*actsComp}), it follows that $\N$ is invariant under conjugation by $H$. Thus,  $n^{xh}=(n^x)^h\in\N$. This proves \eqref{E:NunlhdFstarH}.

\smallskip

Since $\hat{\H}$ is a subnormal, it is a regular locality with Sylow subgroup $S_0:=S\cap \hat{\H}$. Note that $H\cap S_0=H\cap S\in\Syl_p(H)$, as $H\subn N_\L(T^*)$ and $S\in\Syl_p(N_\L(T^*))$. Moreover, since $E(\hat{\H})=E(\L)$, it follows from  Lemma~\ref{L:NLT*} that $N_{\hat{\H}}(F^*(\hat{\H})\cap S)=N_{\hat{\H}}(E(\hat{\H})\cap S)=N_{\hat{\H}}(E(\L)\cap S)=N_{\hat{\H}}(T^*)$ and thus $H\subn N_{\hat{\H}}(F^*(\hat{\H})\cap S)$. The property \eqref{E:NunlhdFstarH} allows us now to apply Lemma~\ref{L:2a} with $\hat{\H}$ in place of $\L$ to obtain that
\[\H:=\N H=H\N\]
is a partial subgroup of $\L$ with $\N H\cap S=\N H\cap S_0=(\N\cap S_0)(H\cap S_0)=(\N\cap S)(H\cap S)$. We show next that
\begin{equation}\label{E:HunlhdhH}
\H\unlhd \hat{\H}\mbox{ and }\H\subn\L.
\end{equation}
As $\hH\subn\L$, it is indeed sufficient to prove that $\H\unlhd\hH$. For the proof fix $a\in \H$ and $f\in\hat{\H}$ such that $w:=(f^{-1},a,f)\in\D$. As $\H$ is a partial subgroup, we only need to prove that $a^f\in \H$. Write $f=yh$ with $y\in E(\L)$ and $h\in H$. Then \cite[Lemma~1.4(f)]{Chermak:2015} gives $f^{-1}=h^{-1} y^{-1}$. Setting $w':=(h^{-1},y^{-1},a,y,h)$, it follows from Lemma~\ref{L:1} that $S_{w'}\cap T^*=S_w\cap T^*$. Hence, \eqref{E:SwcapTstar} gives first $S_w\cap T^*\in\Delta$ and then $w'\in\D$. Using the axioms of a partial group, we can thus conclude that
\[a^f=\Pi(w)=\Pi(w')=(a^y)^h.\]
By part (a) of Lemma~\ref{L:ComponentProducts}, there exist $n\in\N$ and $m\in\M$ with $y=nm$. Moreover, part (b) of that lemma yields then
$S_y\cap T^*=S_{nm}\cap T^*=S_{(n,m)}\cap T^*$. So by \eqref{E:SwcapTstar}, we have $(m^{-1},n^{-1},a,n,m)\in\D$ via $S_{(y^{-1},a,y)}\cap T^*$ and thus
\[a^y=(a^n)^m.\]
Notice that $a,n\in \N H=\H$ and thus $a^n\in \H$. Lemma~\ref{L:ComponentProducts}(b) gives that $\N\subseteq C_\L(\M)$. By assumption we have moreover $H\subseteq C_\L(\M)$. As $\M\unlhd F^*(\L)$ is subnormal in $\L$ by \cite[Proposition~11.7]{Henke:Regular}, it follows from \cite[Theorem~A(f)]{Henke:NormalizerSubnormal} that $C_\L(\M)$ is a partial subgroup of $\L$. Hence, $\H=\N H\subseteq C_\L(\M)$ and thus $\M\subseteq C_\L(\H)$ by \cite[Lemma~3.5]{Henke:Regular}. It follows that $a^y=(a^n)^m=a^n\in\H$. Recall that $h\in H\subseteq\H$. Hence, $a^f=(a^y)^h\in \H$. This completes the proof of \eqref{E:HunlhdhH}. We prove next that
\begin{equation}\label{E:NHofScapN}
N_\H(S\cap \N)=N_\N(S\cap \N)H\mbox{ and }N_\H(S\cap\N)=N_\H(T^*)\subn N_\L(T^*)
\end{equation}
To see this notice that $H\subseteq N_\H(T^*)\subseteq N_\H(T^*\cap\N)=N_\H(S\cap \N)$ as $\N\unlhd\hH\supseteq\H$ by \eqref{E:NunlhdFstarH}. Hence, by the Dedekind Lemma \cite[Lemma~1.10]{Chermak:2015}, we have $N_\H(S\cap \N)=N_\N(S\cap \N)H$. Applying Lemma~\ref{L:NHT*} with $\N$ in place of $\H$ and noting that $\N=E(\N)$, we obtain $N_\N(S\cap\N)=N_\N(T^*)$. Hence, $N_\H(S\cap \N)\leq N_\H(T^*)$. This shows  $N_\H(S\cap\N)=N_\H(T^*)$. As $\H\subn\L$, one sees easily that $N_\H(T^*)\subn N_\L(T^*)$. This completes the proof of \eqref{E:NHofScapN}. It remains now only to prove that
\begin{equation}\label{E:CompH}
\Comp(\H)=\fC.
\end{equation}
Note first that, by \cite[Remark~11.2(a)]{Henke:Regular}, we have $\fC\subseteq\Comp(\H)$. Assuming that \eqref{E:CompH} is false, there exists thus $\K\in\Comp(\H)\backslash\fC$. Then $\K$ centralizes $\N$ by \cite[Theorem~11.18(e)]{Henke:Regular}. In particular, $\K\subseteq C_\H(S\cap\N)\subseteq N_\H(S\cap\N)$. Using  \eqref{E:NHofScapN} and the fact that $\K$ is subnormal in $\H$, we see then that $\K\subn N_\H(S\cap\N)\subn N_\L(T^*)$ and so $\K$ is a subnormal subgroup of the group $N_\L(T^*)$. As $N_\L(T^*)$ is a group of characteristic $p$, it follows from \cite[Lemma~1.2(a)]{MS:2012b} that $\K$ is a group of characteristic $p$. As $\K$ is quasisimple, \cite[Lemma~7.10]{Henke:Regular} gives $Z(\K)=O_p(\K)\neq \K$. This yields a contradiction. Hence, \eqref{E:CompH} holds. In particular, $E(\H)=\N$ and the proof is complete.
\end{proof}

\subsection{The proof of Theorem~\ref{T:RegularLocalities} and related results}

In this subsection we prove Theorem~\ref{T:RegularLocalities} as well as some additional properties. Except in Corollary~\ref{C:RegularLocalities}, we assume throughout this subsection the following hypothesis.

\begin{hypo}\label{H:Regular}
Let $(\L,\Delta,S)$ be a regular locality and set $T^*:=F^*(\L)\cap S$. Let $\H_1$ and $\H_2$ be partial subnormal subgroups of $\L$. Set
\[\H:=\<\H_1,\H_2\>,\]
\[S_i:=\H_i\cap S,\;T_i=E(\H_i)\cap S\mbox{ and }H_i:=N_{\H_i}(T_i)\mbox{ for }i=1,2.\]
Set moreover
\[H:=\<H_1,H_2\>,\;\fC:=\Comp(\H_1)\cup\Comp(\H_2),\;\N:=\prod_{\K\in\fC}\K\mbox{ and }\M:=\prod_{\K\in\Comp(\L)\backslash\fC}\K.\]
\end{hypo}

It follows from \cite[Remark~11.2(b)]{Henke:Regular} that the $\fC\subseteq\Comp(\L)$. In particular, $T_i\leq E(\L)\cap S\leq T^*$ for $i=1,2$. Moreover, as remarked before, $\N$ and $\M$ are well-defined by \cite[Proposition~11.7]{Henke:Regular} (i.e. the order of the factors in these products does not matter). We will use these properties throughout without further reference.

\begin{lemma}\label{L:ab}
We have $H_i=N_{\H_i}(T^*)\subn N_\L(T^*)$ and $S_i=H_i\cap S$. In particular,
\[H=\<H_1,H_2\>\subn N_\L(T^*)\mbox{ and }H\cap S=\<S_1,S_2\>.\]
\end{lemma}

\begin{proof}
Lemma~\ref{L:NHT*} implies that $H_i=N_{\H_i}(T^*)\subn N_\L(T^*)$ and $S_i=S\cap H_i$ for each $i=1,2$. In particular, it follows from Wielandt's Join Theorem for groups that $H=\<H_1,H_2\>$ is a subnormal subgroup of $N_\L(T^*)$ and from  Meierfrankenfeld's Lemma~\ref{L:Meierfrankenfeld} that $H\cap S=\<S_1,S_2\>$.
\end{proof}

Recall that $N_\L(T^*)$ (and thus also $H$) acts on the set of components of $\L$ by Lemma~\ref{L:NLT*actsComp}.

\begin{lemma}\label{L:CHinvariant}
The set $\fC$ is $H$-invariant.
\end{lemma}

\begin{proof}
It is sufficient to argue that $\fC$ is $H_i$-invariant for each $i=1,2$. For the proof fix $i\in\{1,2\}$. By Lemma~\ref{L:NLT*actsComp} (applied with $\H_i$ in place of $\L$), $\Comp(\H_i)$ is $H_i$-invariant. Moreover, applying first  \cite[11.17]{Henke:Regular} and then  \cite[Lemma~3.5]{Henke:Regular}, one sees that $\H_i\subseteq C_\L(\K)$ for every $\K\in\Comp(\L)\backslash\Comp(\H_i)$. In particular, $H_i$ centralizes every component in $\fC\backslash \Comp(\H_i)$. Thus $\fC$ is $H_i$-invariant.
 \end{proof}

\begin{lemma}\label{L:HinCM}
We have $\H\subseteq C_\L(\M)$ and so $H\subseteq C_\L(\M)$.
\end{lemma}

\begin{proof}
By \cite[Proposition~11.7]{Henke:Regular}, $\M\unlhd F^*(\L)$ is subnormal in $\L$. Hence, it follows from \cite[Theorem~A(f)]{Henke:NormalizerSubnormal} that $C_\L(\M)$ is a partial subgroup of $\L$. Therefore, it is sufficient to show that $\H_i\subseteq C_\L(\M)$ for each $i=1,2$.

\smallskip

Fix now $i\in\{1,2\}$ and notice that $\fC':=\Comp(\L)\backslash\fC\subseteq \Comp(\L)\backslash\Comp(\H_i)$. Hence, it is a special case of \cite[Theorem~11.18(c)]{Henke:Regular} that $\H_i\M$ is an internal central product of the elements of $\{\H_i\}\cup\fC'$ (in the sense defined in \cite[Definition~4.1]{Henke:Regular}). In particular, by \cite[Lemma~4.8]{Henke:Regular}, we have $\H_i\subseteq C_\L(\M)$ as required.
\end{proof}

Note that Theorem~\ref{T:RegularLocalities} is implied by the following theorem.

\begin{theorem}\label{T:c}
We have $\H:=\<\H_1,\H_2\>=\N H=H\N\subn\L$. Moreover, $\H\cap S=H\cap S=\<S_1,S_2\>$, $\Comp(\H)=\fC$, $E(\H)=\N=E(\H_1)E(\H_2)\unlhd \H$ and
\[N_\H(F^*(\H)\cap S)=N_\H(\N\cap S)=N_\H(T^*)=H.\]
\end{theorem}

\begin{proof}
Recall that $H\subn N_\L(T^*)$ by Lemma~\ref{L:ab}, that $\fC$ is $H$-invariant by Lemma~\ref{L:CHinvariant} and that $H\subseteq C_\L(\M)$ by Lemma~\ref{L:HinCM}. Hence, it follows from Lemma~\ref{L:3} that $\N H=H\N$ is a partial subnormal subgroup of $\L$,
\[\Comp(\N H)=\fC,\;E(\N H)=\N\unlhd\N H,\;(\N H)\cap S=(\N\cap S)(H\cap S)\mbox{ and }N_{\N H}(\N\cap S)=N_\N(T^*)H.\]
It follows from the Frattini Lemma \cite[Corollary~3.11]{Chermak:2015} that $\H_i=E(\H_i)H_i\subseteq \N H$ for every $i\in\{1,2\}$. So the fact that $\N H$ is a partial subgroup implies $\H:=\<\H_1,\H_2\>\subseteq \N H$. As $\N$ and $H$ are both contained in $\H$, we can conclude that $\H=\N H=H\N$. In particular, $\H$ is subnormal in $\L$,
\[\Comp(\H)=\fC,\;E(\H)=\N\unlhd\H,\;\H\cap S=(\N\cap S)(H\cap S)\mbox{ and }N_\H(\N\cap S)=N_\N(T^*)H.\]
By Lemma~\ref{L:ab}, $H\cap S=\<S_1,S_2\>$. It follows moreover from Lemma~\ref{L:ComponentProducts}(a) that
\[\N\cap S=(E(\H_1)\cap S)(E(\H_2)\cap S)\leq \<S_1,S_2\>=H\cap S.\] Hence, $\H\cap S=(\N\cap S)(H\cap S)=H\cap S=\<S_1,S_2\>$. Notice that Lemma~\ref{L:ComponentProducts}(a) yields also that $\N=E(\H_1)E(\H_2)$, i.e. $E(\H)=\N=E(\H_1)E(\H_2)$.

\smallskip

As $\N=E(\H)$, it follows from Lemma~\ref{L:NLT*} that $N_\H(S\cap F^*(\H))=N_\H(S\cap\N)$ and from Lemma~\ref{L:NHT*} that $N_\H(\N\cap S)=N_\H(T^*)$.  It remains thus only to show that $N_\H(T^*)=H$.
By Lemma~\ref{L:ComponentProducts}(c), we have $N_\N(T^*)=N_{E(\H_1)}(T^*)N_{E(\H_2)}(T^*)$. Note that $N_{E(\H_i)}(T^*)\leq N_{\H_i}(T^*)=H_i$ for each $i=1,2$ by Lemma~\ref{L:ab}. Thus, it follows that $N_\N(T^*)\leq \<H_1,H_2\>=H$. Using the properties above, we obtain therefore that $N_\H(T^*)=N_\H(\N\cap S)=N_\N(T^*)H=H$. This completes the proof.
\end{proof}

\begin{lemma}\label{L:d}
Set $T:=\H\cap S=H\cap S$. Then
\[\F_T(\H)=\<\F_T(\N T),\F_T(H)\>.\]
\end{lemma}

\begin{proof}
Recall that $\N\unlhd \H\subn\L$ and $N_\H(S\cap \N)=H$ by Theorem~\ref{T:c}. In particular, by Theorem~\ref{T:RegularSubnormal}, $\H$ supports the structure of a regular locality. Hence, the assertion follows from Lemma~\ref{L:SplitFSL}(b) applied with $\H$ in place of $\L$.
\end{proof}

When we show Theorem~\ref{T:FusionSystems} using Theorem~\ref{T:RegularLocalities}, we need the following lemma. Note that its proof relies on Theorem~\ref{T:Groups}.

\begin{lemma}\label{L:ApplyThmC}
Assume that $T:=S\cap \H\leq N_S(\H_1)\cap N_S(\H_2)$. Then
\[\F_T(\H)=\<\F_T(\H_1T),\F_T(\H_2T)\>.\]
\end{lemma}

\begin{proof}
As $T\leq N_S(\H_i)$ and $E(\H_i)$ is by \cite[Lemma~11.12]{Henke:Regular} invariant under automorphisms of $\H_i$, we have
\begin{equation*}
 T\leq N_S(E(\H_i))\mbox{ and }T_i\unlhd T\mbox{ for }i=1,2.
\end{equation*}
Let $i\in\{1,2\}$. By \cite[Lemma~3.18(a)]{Henke:NormalizerSubnormal}, $\H_i T$ is a partial subgroup of $\L$ and $(\H_iT,\delta(\F_T(\H_iT)),T)$ is a regular locality with $E(\H_i)=E(\H_iT)$. In particular, $E(\H_i)=E(\H_iT)\unlhd\H_iT$ by \cite[Lemma~11.13]{Henke:Regular}. It follows thus from Lemma~\ref{L:SplitFSL}(b)  that
\[\F_T(\H_iT)=\<\F_T(E(\H_i)T),\F_T(N_{\H_iT}(T_i))\>\mbox{ for }i=1,2.
\]
By the Dedekind Lemma \cite[Lemma~1.10]{Chermak:2015}, $N_{\H_iT}(T_i)=N_{\H_i}(T_i)T=H_iT$. Hence,
 \begin{equation}\label{E:FTHiT}
 \F_T(\H_iT)=\<\F_T(E(\H_i)T),\F_T(H_iT)\>\mbox{ for }i=1,2.
\end{equation}
As $\N\unlhd \H$ by Theorem~\ref{T:c}, the product $\N T$ is a partial subgroup of $\L$ by \cite[Lemma~3.15]{Henke:Regular}. We show next that
\begin{equation}\label{E:EHiunlhdNT}
E(\H_i)\unlhd \N T \mbox{ for each }i=1,2.
\end{equation}
For the proof let $i\in\{1,2\}$, $x\in E(\H_i)$ and $f\in \N T$ with $u:=(f^{-1},x,f)\in\D$. Then there exist $n\in\N$ and $s\in T$ with $f=ns$. By \cite[Lemma~1.4(f)]{Chermak:2015}, $f^{-1}=s^{-1}n^{-1}$. Moreover, by \cite[Lemma~2.8]{Henke:2020}, we have $S_f=S_{(n,s)}$ and $S_{f^{-1}}=S_{(s^{-1},n^{-1})}$. Hence, for $v:=(s^{-1},n^{-1},x,n,s)$, we have $S_v=S_u$. Applying \eqref{E:Sw} twice, it follows first that $S_v=S_u\in\Delta$ and then that $v\in\D$. Hence, by the axioms of a partial group, we have $x^f=\Pi(u)=\Pi(v)=(x^n)^s$. As $E(\H_i)\unlhd F^*(\L)\supseteq \N$ by \cite[Proposition~11.7]{Henke:Regular}, we have $x^n\in E(\H_i)$. So $s\in T\leq N_S(E(\H_i))$ implies that $x^f=(x^n)^s\in E(\H_i)$. This shows \eqref{E:EHiunlhdNT}. We show next:
\begin{equation}\label{E:FTNT}
\F_T(\N T)=\<\F_T(E(\H_1)T),\F_T(E(\H_2)T)\>
\end{equation}
As $E(\H_i)T\subseteq \N T$ for $i=1,2$, we have clearly $\F_0:=\<\F_T(E(\H_1)T),\F_T(E(\H_2)T)\>\subseteq\F_T(\N T)$ and so it remains to show the opposite inclusion. Thus, fixing $f\in\N T$, we need to show that $c_f|_{S_f\cap T}$ is a morphism in $\F_0$. Write $f=nt$ for some $n\in\N$ and $t\in T$. By \cite[Lemma~2.8]{Henke:2020}, we have $S_f=S_{(n,t)}$, which implies that $c_f\colon S_f\cap T\rightarrow T$ is a composite of restrictions of $c_n|_{S_n\cap T}$ and $c_t|_T$. As $c_t|_T$ is a morphism in $\F_0$, it is thus sufficient to show that $c_n|_{S_n\cap T}$ is a morphism in $\F_0$. By \cite[Lemma~3.18(a)]{Henke:NormalizerSubnormal} $(\N T,\delta(\F_T(\N T)),T)$ is a regular locality. Moreover, by Theorem~\ref{T:c}, $\N=E(\H_1)E(\H_2)$ and by \eqref{E:EHiunlhdNT}, $E(\H_i)\unlhd \N T$ for each $i=1,2$. Therefore, \cite[Theorem~1]{Henke:2015} applied with $\N T$ in place of $\L$ yields the existence of elements $x\in E(\H_1)$ and $y\in E(\H_2)$ with $(x,y)\in\D$, $n=xy$ and $S_n\cap T=S_{(x,y)}\cap T$. So $c_n|_{S_n\cap T}$ is the composite of restrictions of $c_x|_{S_x\cap T}$ (which is a morphism in $\F_T(E(\H_1)T)$) and of $c_y|_{S_y\cap T}$ (which is a morphism in $\F_T(E(\H_2)T)$). So $c_n|_{S_n\cap T}$ is a morphism in $\F_0$ and \eqref{E:FTNT} holds.

\smallskip

We use now that $T=\H\cap S=H\cap S$ by Theorem~\ref{T:c}. Recall also that  $H_1$ and $H_2$ are subnormal in $N_\L(T^*)$ by Lemma~\ref{L:ab}. Our assumption yields moreover that $H_1$ and $H_2$ are $T$-invariant and so $\<H_i,T\>=H_iT$ for $i=1,2$. Hence, it follows from Theorem~\ref{T:Groups} that
\[\F_T(H)=\<\F_T(H_1 T),\F_T(H_2 T)\>.\]
So Lemma~\ref{L:d} implies
\[\F_T(\H)=\<\F_T(\N T),\F_T(H)\>=\<\F_T(\N T),\F_T(H_1 T),\F_T(H_2 T)\>.\]
Hence, the assertion follows from \eqref{E:FTHiT} and \eqref{E:FTNT}.
\end{proof}

We remove now the standing Hypothesis~\ref{H:Regular} to record the following corollary to Theorem~\ref{T:RegularLocalities}.

\begin{corollary}\label{C:RegularLocalities}
Let $(\L,\Delta,S)$ be a regular locality and suppose $\H_1,\H_2,\dots,\H_n$ are partial subnormal subgroups of $\L$. Then $\<\H_1,\H_2,\dots,\H_n\>$
is a partial subnormal subgroup of $\L$ with
\[\<\H_1,\H_2,\dots,\H_n\>\cap S=\<\H_1\cap S,\H_2\cap S,\dots,\H_n\cap S\>.\]
\end{corollary}

\begin{proof}
This follows from Theorem~\ref{T:RegularLocalities} using induction on $n$.
\end{proof}

\section{Wielandt's Join Theorem for fusion systems and related results}\label{S:FusionSystems}

After collecting some background, we prove Theorem~\ref{T:FusionSystems} in this section. Indeed, Theorem~\ref{T:FusionSystemsn} below gives some additional information. Some of the difficulties in formulating Wielandt's Join Theorem for fusion systems are illustrated in Example~\ref{E:1}. At the end we prove Proposition~\ref{P:GroupFusionBracket}.

\subsection{Some background} Throughout this subsection let $\F$ be a saturated fusion system over $S$. Let $\E$ be a subsystem of $\F$ over $T\leq S$. Given $\alpha\in\Hom_\F(T,S)$,  write $\E^\alpha$ for the subsystem of $\F$ over $T\alpha$ such that $\Hom_{\E^\alpha}(P\alpha,Q\alpha)=\{\alpha^{-1}\phi\alpha\colon \phi\in\Hom_\E(P,Q)\}$ for all $P,Q\leq T$. For $a\in S$ set $\E^a:=\E^\alpha$ where $\alpha=c_a$ is the conjugation map $T\rightarrow S$. Set
\[N_S(\E):=\{a\in N_S(T)\colon \E^a=\E\}.\]
One observes easily that $N_S(\E)$ is a subgroup of $S$.

\smallskip

The reader might want to recall the definition of $O^p(\F)$ from  \cite[Definition~I.7.3, Theorem~I.7.4]{Aschbacher/Kessar/Oliver:2011}. If $\E$ is a subnormal subsystem of $\F$ over $T\leq S$ and $P\leq N_S(\E)$, then a concrete description of a subsystem
$\E P=(\E P)_\F$ of $\F$ is given in \cite[Definition~2.7]{Henke:NormalizerSubnormal}. The subsystem $\E P$ should be thought of as a product of $\E$ with $P$. It depends however not only on $\E$ and $P$, but also on $\F$. We write  $(\E P)_\F$ if we want to emphasize that dependence. It is shown in \cite[Theorem~B(a)]{Henke:NormalizerSubnormal} that $\E P=(\E P)_\F$ is the unique saturated subsystem of $\F$ over $TP$ with $O^p(\E P)=O^p(\E)$. This characterization implies immediately the following remark:

\begin{remark}\label{R:Product}
Let $\E$ be a subnormal subsystem of the saturated fusion system $\F$, and let $\G$ be a saturated subsystem of $\F$ over $S'\leq S$ such that $\E$ is also  subnormal in $\G$. If $P\leq N_{S'}(\E)$, then $(\E P)_\G=(\E P)_\F$.
\end{remark}

The following lemma will be used in the proof of Proposition~\ref{P:GroupFusionBracket}. Part (a) goes back to  Puig.

\begin{lemma}\label{L:Hyperfocal}
Let $G$ be a finite group, $S\in\Syl_p(G)$ and $\F=\F_S(G)$. Then the following hold:
\begin{itemize}
 \item [(a)] $\F_{S\cap O^p(G)}(O^p(G))=O^p(\F)$.
 \item [(b)] Let $H\subn G$, $T=S\cap H$ and $\E:=\F_T(H)$. Then $\E\subn\F$ and $N_S(H)\leq N_S(\E)$. Moreover, $(\E P)_\F=\F_{TP}(HP)$ for every $P\leq N_S(H)$.
\end{itemize}
\end{lemma}

\begin{proof}
\textbf{(a)} It was first observed by Puig \cite[$\S$ 1.1]{Puig:2000} that $S\cap O^p(G)=\hyp(\F)$; a detailed proof can be found in \cite[Theorem~1.33]{Craven:Book}. Thus $\F_{S\cap O^p(\F)}(O^p(\F))$ is a saturated subsystem of $\F$ over $\hyp(\F)$. As $O^p(\F)$ is characterized in \cite[Theorem~I.7.4]{Aschbacher/Kessar/Oliver:2011} as the unique saturated subsystem of $\F$ over $\hyp(\F)$, part (a) follows.

\smallskip

\textbf{(b)} It follows from \cite[Proposition~I.6.2]{Aschbacher/Kessar/Oliver:2011} that $\E$ is subnormal in $\F$, and one observes easily that $N_S(H)\leq N_S(\E)$. For $P\leq N_S(H)$ notice that $O^p(HP)=O^p(H)$ and so (a) yields $O^p(\F_{TP}(HP))=\F_{O^p(H)\cap S}(O^p(H))=O^p(\E)$. As $\E P=(\E P)_\F$ is by \cite[Theorem~B(a)]{Henke:NormalizerSubnormal} the unique saturated subsystem of $\F$ over $TP$ with $O^p(\E P)=O^p(\E)$, statement (b) follows.
\end{proof}

\subsection{Wielandt's Join Theorem for fusion systems}

In this subsection we assume the following hypothesis:

\begin{hypo}\label{H:FusionSystems}
Let $\F$ be a saturated fusion system and $n\geq 1$. For $i=1,2,\dots,n$ let $\E_i$ be a subnormal subsystem of $\F$ over $S_i\leq S$.
\end{hypo}

We will show the following theorem, which implies Theorem~\ref{T:FusionSystems}.

\begin{theorem}\label{T:FusionSystemsn}
Assume Hypothesis~\ref{H:FusionSystems}. Then there exists a subsystem
\[\E:=\la\E_1,\E_2,\dots,\E_n\ra\]
of $\F$ over $T:=\<S_1,S_2,\dots,S_n\>$ such that the following hold.
\begin{itemize}
 \item [(a)] $\E$ is subnormal in $\F$ and $\E_i\subn\E$ for $i=1,2,\dots,n$.
 \item [(b)] If $\G$ is a saturated subsystem of $\F$ with $\E_i\subn\G$ for $i=1,2,\dots,n$, then $\E\subn\G$. In particular, $\E$ is the smallest saturated subsystem of $\F$ in which $\E_1,\dots,\E_n$ are subnormal.
 \item [(c)] Let $0=i_0< i_1<i_2<\cdots <i_k=n$. Then
 \[\E=\la \;\la \E_1,\dots\E_{i_1}\ra,\la \E_{i_1+1},\dots,\E_{i_2}\ra,\dots,\la \E_{i_{k-1}+1},\dots ,\E_{i_k}\ra\;\ra\]
 (where for $j=1,\dots,k$, $\la \E_{i_{j-1}+1},\dots,\E_{i_j}\ra$ is the smallest saturated subsystem in which $\E_{i_{j-1}+1},\dots,\E_{i_j}$ are subnormal).
 \item [(d)] Let $(\L,\Delta,S)$ be a regular locality over $\F$. For $i=1,2,\dots,n$ let $\H_i\subn\L$ with $S_i=\H_i\cap S$ and $\F_{S_i}(\H_i)=\E_i$. Then, setting $\H:=\<\H_1,\H_2,\dots,\H_n\>$, we have
 \[T=\H\cap S\mbox{ and }\E=\F_T(\H).\]
\end{itemize}
\end{theorem}

The remainder of this section is devoted to the proof of Theorem~\ref{T:FusionSystemsn}. For that we fix a regular locality $(\L,\Delta,S)$ over $\F$. Such a regular locality exists always by \cite[Lemma~10.4]{Henke:Regular}. By \cite[Theorem~F]{Chermak/Henke}, for each $i=1,2,\dots,n$, there exists a unique partial subnormal subgroup $\H_i$ of $\L$ with $\H_i\cap S=S_i$ and $\F_{S_i}(\H_i)=\E_i$. Set now
\[T:=\<S_1,S_2,\dots,S_n\>\mbox{ and }\H:=\<\H_1,\H_2,\dots,\H_n\>.\]
Moreover, define
\[\E:=\la \E_1,\E_2,\dots,\E_n\ra:=\F_{\H\cap S}(\H).\]
Similarly, we define $\la \mD_1,\dots,\mD_r\ra$ whenever $\mD_1,\dots,\mD_r$ are subnormal subsystems of $\F$. Note that the definition depends a priori on the choice of the regular locality $(\L,\Delta,S)$. However, we will show below that Theorem~\ref{T:FusionSystemsn}(b) holds for this choice of $\E$, and thus $\E$ is in fact uniquely determined by $\F$ and $\E_1,\dots,\E_n$.

\begin{lemma}\label{L:FusionSystemsFirstPart}
The following hold:
\begin{itemize}
 \item [(a)] $\H\cap S=T$ and $\E$ is a subnormal subsystem of $\F$ over $T$. Moreover, $\E_i\subn\E$ for all $i=1,2,\dots,n$.
 \item [(b)] Let $1\leq i_1<i_2<\cdots <i_k=n$. Then
 \[\E=\la \;\la \E_1,\dots\E_{i_1}\ra,\la \E_{i_1+1},\dots,\E_{i_2}\ra,\dots,\la \E_{i_{k-1}+1},\dots ,\E_{i_k}\ra\;\ra.\]
\end{itemize}
\end{lemma}

\begin{proof}
\textbf{(a)} By Corollary~\ref{C:RegularLocalities}, we have $\H\cap S=T$ and $\H\subn\L$. In particular, $\E=\F_{\H\cap S}(\H)=\F_T(\H)$ is a subsystem over $T$. Moreover, it follows from \cite[Proposition~7.1(a)]{Chermak/Henke} that $\E$ is subnormal in $\F$. As $\H_i\subseteq\H$ and $\H_i\subn\L$, it is a consequence of \cite[Proposition~7.1(c)]{Chermak/Henke} that $\E_i=\F_{S\cap\H_i}(\H_i)\subn\F_T(\H)=\E$ for all $i=1,2,\dots,n$.

\smallskip

\textbf{(b)} This follows since
\[\H:=\<\H_1,\H_2,\dots,\H_n\>=\<\;\<\H_1,\dots\H_{i_1}\>,\<\H_{i_1+1},\dots,\H_{i_2}\>,\dots,\< \H_{i_{k-1}+1},\dots ,\H_{i_k}\>\;\>.\]
\end{proof}

The next goal is the proof of Theorem~\ref{T:FusionSystemsn}(b). We start with two preliminary results, which will also be used in the proof of Proposition~\ref{P:GroupFusionBracket}.

\begin{lemma}\label{L:n=2Help}
Suppose $n=2$ and $T\leq N_S(\E_1)\cap N_S(\E_2)$. Then
\[\E=\<(\E_1T)_\F,(\E_2T)_\F\>.\]
\end{lemma}

\begin{proof}
We use throughout that $T=\H\cap S$ by Lemma~\ref{L:FusionSystemsFirstPart}(a).
By \cite[Lemma~3.8]{Henke:NormalizerSubnormal}, $N_S(\E_i)=N_S(\H_i)$ for $i=1,2$. Hence, $T\leq N_S(\H_1)\cap N_S(\H_2)$ and so  Lemma~\ref{L:ApplyThmC} gives
\[\E=\<\F_T(\H_1T),\F_T(\H_2T)\>.\]
By \cite[Theorem~B(b)]{Henke:NormalizerSubnormal}, $\F_T(\H_iT)=(\E_iT)_\F$ for $i=1,2$. Hence the assertion holds.
\end{proof}

\begin{lemma}\label{L:Conjugate}
Let $x\in T$. Then $\E_1^x\subn\F$ and $\E=\la \E_1,\E_1^x,\E_2,\E_3,\dots,\E_n\ra$.
\end{lemma}

\begin{proof}
Since $c_x\in \Aut(S)$ induces an automorphism of $\F$, we have $\E_1^x\subn\F$. Indeed, it is shown in \cite[Lemma~3.26(a)]{Henke:NormalizerSubnormal} that $\H_1^x\subn\L$ and $\E_1^x=\F_{S\cap \H_1^x}(\H_1^x)$ (which implies by \cite[Theorem~7.1(a)]{Chermak/Henke} also that $\E_1^x\subn\F$). As $x\in T\subseteq\H=\<\H_1,\H_2,\dots,\H_n\>$, we have moreover that $\H=\<\H_1,\H_1^x,\H_2,\dots,\H_n\>$. It follows now from the definitions of $\E$ and of $\la \E_1,\E_1^x,\E_2,\dots,\E_n\ra$ that
\[\E=\F_{\H\cap S}(\H)=\la\E_1,\E_1^x,\E_2,\dots,\E_n\ra\]
\end{proof}

\begin{lemma}\label{L:EsubnG}
Let $\G$ be a saturated subsystem of $\F$ with $\E_i\subn\G$ for $i=1,2,\dots,n$. Then $\E\subn\G$.
\end{lemma}

\begin{proof}
If $n=1$, then $\E=\E_1\subn\G$. Hence, using   Lemma~\ref{L:FusionSystemsFirstPart}(b) and induction on $n$, we can reduce to the case $n=2$. Thus, we assume from now on that
\[n=2.\]
Suppose moreover that $(\F,\G,\E_1,\E_2,\L,\Delta,S)$ is a counterexample such that first $|S:T|$ and then $|\E_1|+|\E_2|$ is maximal, where $T=\<S_1,S_2\>$ as before and $|\E_i|$ denotes the number of morphisms in $\E_i$ for $i=1,2$.

\smallskip

Let $S'\leq S$ such that $\G$ is a subsystem over $S'$. As $\E_1$ and $\E_2$ are contained in $\G$, we have $T\leq S'$. Fix now a regular locality $(\L',\Delta',S')$ over $\G$ (which exists by \cite[Lemma~10.4]{Henke:Regular}). By \cite[Theorem~E(a)]{Chermak/Henke}, there exist partial subnormal subgroups $\H_1',\H_2'$ of $\L'$ such that $S_i=\H_i'\cap S'$ and $\E_i=\F_{S_i}(\H_i')$ for $i=1,2$. So we may define
\[\la \E_1,\E_2\ra_\G:=\F_{S\cap \H'}(\H')\mbox{ where }\H':=\<\H_1',\H_2'\>.\]
Notice that $\la \E_1,\E_2\ra_\G\subn\G$ by Lemma~\ref{L:FusionSystemsFirstPart}(a) applied with $(\G,\L',\Delta',S')$ in place of $(\F,\L,\Delta,S)$.

\smallskip

Assume now first that $T\leq N_S(\E_1)\cap N_S(\E_2)$. Then Lemma~\ref{L:n=2Help} yields $\E=\< (\E_1T)_\F,(\E_2T)_\F\>$. As  $T\leq S'$, we have $T\leq N_{S'}(\E_1)\cap N_{S'}(\E_2)$. Hence, Lemma~\ref{L:n=2Help} applied with $(\G,\L',\Delta',S')$ in place of $(\F,\L,\Delta,S)$ yields also that $\la \E_1,\E_2\ra_\G=\<(\E_1 T)_\G,(\E_2 T)_\G\>$. It follows now from Remark~\ref{R:Product} that $(\E_i T)_\G=(\E_i T)_\F$ for $i=1,2$ and thus
\[\E=\la \E_1,\E_2\ra_\G\subn\G.\]
This contradicts the assumption that $(\F,\G,\E_1,\E_2,\L,\Delta,S)$ is a counterexample. Hence, $T\not\leq N_S(\E_i)$ for some $i=1,2$. Since the situation is symmetric in $\E_1$ and $\E_2$, we may assume without loss of generality that
\[T\not\leq N_S(\E_1).\]
This means that $T_0:=T\cap N_S(\E_1)<T$ and thus $T_0<N_T(T_0)$.
Fix $x\in N_T(T_0)\backslash T_0$. By Lemma~\ref{L:Conjugate}, we have $\E_1^x\subn\F$ and $\E_1^x\subn\G$. Set now
\[\tE_1:=\la \E_1,\E_1^x\ra.\]
Notice that $\E_1^x$ is a subsystem over $S_1^x$ and $\tS_1:=\<S_1,S_1^x\>\leq T_0<T$ as $S_1\leq T_0$. Thus, $|S:\tS_1|>|S:T|$, and so the maximality of $|S:T|$ yields that $(\F,\G,\E_1,\E_1^x,\L,\Delta,S)$ is not a counterexample. Hence,
\[\tE_1\subn \G.\]
By Lemma~\ref{L:FusionSystemsFirstPart}(a), $\tE_1$ is a subnormal subsystem of $\F$ over $\tS_1$. As $x\in T$ and $\tS_1=\<S_1,S_1^x\>$, it follows that $T=\<S_1,S_2\>=\<\tS_1,S_2\>$. The choice of $x$ yields that $\E_1^x\neq \E_1$. Since $\E_1$ and $\E_1^x$ are contained in $\tE_1$, the subsystem $\E_1$ is therefore properly contained in $\tE_1$. So $|\tE_1|>|\E_1|$ and the maximality of $|\E_1|+|\E_2|$ yields that $(\F,\G,\tE_1,\E_2,\L,\Delta,S)$ is not a counterexample. As $\tE_1$ and $\E_2$ are subnormal in $\F$ and in $\G$, this means that
\begin{equation*}
\la \tE_1,\E_2\ra \subn\G.
\end{equation*}
Applying first Lemma~\ref{L:Conjugate} and then Lemma~\ref{L:FusionSystemsFirstPart}(b), we can conclude now that
\[\E=\la\E_1,\E_1^x,\E_2\ra=\la\tE_1,\E_2\ra\subn\G.\]
This contradicts the assumption that $(\F,\G,\E_1,\E_2,\L,\Delta,S)$ is a counterexample and completes thereby the proof.
\end{proof}

\begin{proof}[Proof of Theorem~\ref{T:FusionSystemsn}]
Note that Lemma~\ref{L:FusionSystemsFirstPart}(a) verifies part (a) and that Lemma~\ref{L:EsubnG} verifies part (b), if $\E$ is defined as above (a priori in dependence of $(\L,\Delta,S)$). In particular, $\E$ is the smallest saturated subsystem of $\F$ in which $\E_1,\E_2,\dots,\E_n$ are subnormal, and $\E$ depends in fact only on $\E_1,\dots,\E_n$ and $\F$, but not on the choice of the regular locality $(\L,\Delta,S)$. Therefore,  part (d) follows from the definition of $\E$, and  (c) follows from Lemma~\ref{L:FusionSystemsFirstPart}(b). 
\end{proof}

We end this subsection with an example which helps to motivate why we formulate Theorem~\ref{T:FusionSystems} and Theorem~\ref{T:FusionSystemsn} as we do. More precisely, in our example we illustrate that, for two subnormal subsystems $\E_1$ and $\E_2$ of a saturated fusion system $\F$, the subsystem generated by $\E_1$ and $\E_2$ may not be saturated and thus in particular not subnormal. Moreover, our example shows that there is not necessarily a smallest saturated or a smallest subnormal subsystem of $\F$ \emph{containing} $\E_1$ and $\E_2$. Thus, it is important to characterize $\la \E_1,\E_2\ra$ as the smallest saturated subsystem in which $\E_1$ and $\E_2$ are \emph{subnormal}.

\begin{ex}\label{E:1}
Let $G=G_1\times G_2$ with $G_1\cong G_2\cong A_4$. Setting $T_i:=O_p(G_i)$ for $i=1,2$ and $S:=T_1\times T_2$, we have $S\in\Syl_p(G)$. Set $\F=\F_S(G)$ and $\E_i=\F_{T_i}(G_i)$ for $i=1,2$. Note that $G_i\unlhd G$ and thus $\E_i\unlhd\F$ for $i=1,2$. In particular, $\E_1$ and $\E_2$ are subnormal in $\F$.

\smallskip

If $\E_1$ is contained in a saturated fusion system $\mD$ over $S$, then, by the extension axiom, every element of $\Aut_{\E_1}(S_1)$ extends to a $\mD$-automorphism of $S$, which yields $\Aut_{\mD}(S)\neq \Aut_S(S)$. In particular, the following holds:
\begin{equation}\label{E:AutDS}
\mbox{If $\mD$ is a saturated subsystem of $\F$ containing $\E_1$ and $\E_2$, then $\Aut_{\mD}(S)\neq\Aut_S(S)$.}
\end{equation}
Note that $\Aut_{\<\E_1,\E_2\>}(S)=\Aut_S(S)$. Hence, $\<\E_1,\E_2\>$ is by \eqref{E:AutDS} not saturated (and is so in particular not subnormal in $\F$).

\smallskip

We argue now that there is no smallest saturated subsystem of $\F$ containing $\E_1$ and $\E_2$, and indeed also no smallest subnormal subsystem of $\F$ containing $\E_1$ and $\E_2$.
Fix $d_i\in G_i$ of order $3$ for $i=1,2$. Set $N_1:=S\<d_1d_2\>$, $N_2:=S\<d_1d_2^2\>$ and $\mD_i:=\F_S(N_i)$ for $i=1,2$. Notice that $N_i\unlhd G$ and thus $\mD_i\unlhd\F$ for $i=1,2$. In particular, $\mD_1$ and $\mD_2$ are subnormal in $\F$. Observe also that $\E_1$ and $\E_2$ are contained in $\mD_i$ for $i=1,2$. However, if $\mD$ is a saturated fusion system containing $\E_1$ and $\E_2$, then $\mD$ is not contained in $\mD_1\cap\mD_2$, as otherwise $\Aut_{\mD}(S)=\Aut_{\mD_1\cap\mD_2}(S)=\Aut_S(S)$, contradicting \eqref{E:AutDS}. Thus, there is no smallest saturated and no smallest subnormal subsystem of $\F$ containing $\E_1$ and $\E_2$.

\smallskip

It might also be worth observing that in this example, $\la\E_1,\E_2\ra=\F$. (One can see this by noting that $(G,\delta(\F),S)$ is a regular locality, and so $\la \E_1,\E_2\ra$ is by Theorem~\ref{T:FusionSystemsn}(d) realized by $\<G_1,G_2\>=G$.) So $\la\E_1,\E_2\ra$ is neither contained in $\mD_1$ nor in $\mD_2$, even though both subsystems contain $\E_1$ and $\E_2$.
\end{ex}

\subsection{The proof of Proposition~\ref{P:GroupFusionBracket}}

Throughout this subsection let $\F$ be a saturated fusion system over $S$. If $\E_1,\E_2,\dots,\E_n$ are subnormal subsystems of $\F$, then $\la\E_1,\E_2,\dots,\E_n\ra$ denotes the subsystem of $\F$ which is characterized by Theorem~\ref{T:FusionSystemsn}(b) as the smallest saturated subsystem of $\F$ in which $\E_1,\E_2,\dots,\E_n$ are subnormal.

\begin{lemma}\label{L:n=2Help1}
Let $\E_1$ and $\E_2$ be subnormal subsystems of $\F$ over subgroups $S_1$ and $S_2$ of $S$ respectively. Set $T:=\<S_1,S_2\>$. Then the following hold:
\begin{itemize}
 \item [(a)] If $T\leq N_S(\E_1)\cap N_S(\E_2)$, then $\la \E_1,\E_2\ra =\<(\E_1T)_\F,(\E_2T)_\F\>$.
 \item [(b)] For every $x\in T$, we have $\la \E_1,\E_2\ra=\la \E_1,\E_1^x,\E_2\ra$.
\end{itemize}
\end{lemma}

\begin{proof}
Using Theorem~\ref{T:FusionSystemsn}(d), part (a) is a restatement of Lemma~\ref{L:n=2Help} and part (b) follows from Lemma~\ref{L:Conjugate}.
\end{proof}

We are now able to prove Proposition~\ref{P:GroupFusionBracket} using a similar strategy as in the proof of Lemma~\ref{L:EsubnG}.

\renewcommand{\tH}{\tilde{H}}

\begin{proof}[Proof of Proposition~\ref{P:GroupFusionBracket}]
Set $S_i=S\cap H_i$ and $\E_i:=\F_{S_i}(H_i)$ for $i=1,2$. Set moreover $T:=\<S_1,S_2\>$ and $H:=\<H_1,H_2\>$. Note that $\E_1$ and $\E_2$ are subnormal in $\F$ as stated in Lemma~\ref{L:Hyperfocal}(b). Thus, the statement of the proposition makes sense. Assume now that $(G,S,H_1,H_2)$ is a counterexample to the proposition such that first $|S:T|$ and then $|H_1|+|H_2|$ is maximal.

\smallskip

To start with, assume that $T\leq N_S(H_1)\cap N_S(H_2)$. Then Lemma~\ref{L:Hyperfocal}(b) gives that $T\leq N_S(\E_1)\cap N_S(\E_2)$ and $(\E_iT)_\F=\F_T(H_iT)$ for $i=1,2$.  Thus, Lemma~\ref{L:n=2Help1}(a) yields that
\[\la \E_1,\E_2\ra=\<(\E_1T)_\F,(\E_2T)_\F\>=\<\F_T(H_1T),\F_T(H_2T)\>.\]
Hence, it follows from Theorem~\ref{T:Groups} that $\la\E_1,\E_2\ra=\F_{H\cap S}(H)$, contradicting the assumption that $(G,S,H_1,H_2)$ is a counterexample. Hence, $T\not\leq N_S(H_1)\cap N_S(H_2)$. Since the situation is symmetric in $H_1$ and $H_2$, we may assume that
\[T\not\leq N_S(H_1).\]
This means that $T_0:=N_T(H_1)<T$ and thus $T_0<N_T(T_0)$ as $T$ is a $p$-group. Fix $x\in N_T(T_0)\backslash T_0$. Note that $S_1\leq T_0$ and so $H_1^x\cap S=S_1^x\leq T_0\leq S$. Hence, $\tS_1:=\<S_1,S_1^x\>\leq T_0<T$ and $|S:\tS_1|>|S:T|$. Observe also that $\F_{H_1^x\cap S}(H_1^x)=\E_1^x$. Setting $\tH_1:=\<H_1,H_1^x\>$, the maximality of $|S:T|$ yields thus that
\[\F_{\tH_1\cap S}(\tH_1)=\la \E_1,\E_1^x\ra.\]
Recall that $x$ is chosen such that $H_1\neq H_1^x$ and $x\in T\leq H=\<H_1,H_2\>$. Hence, $H_1<\tH_1$ and $H=\<\tH_1,H_2\>$. By Wielandt's Join Theorem, $\tH_1$ is subnormal in $G$. Therefore, the maximality of $|H_1|+|H_2|$ yields that
\[\F_{H\cap S}(H)=\la \;\F_{\tH_1\cap S}(\tH_1),\F_{H_2\cap S}(H_2)\;\ra=\la\;\la \E_1,\E_1^x\ra,\E_2\;\ra.\]
Using first Theorem~\ref{T:FusionSystemsn}(c) and then Lemma~\ref{L:n=2Help1}(b), we obtain now
\[\F_{H\cap S}(H)=\la\E_1,\E_1^x,\E_2\ra=\la\E_1,\E_2\ra,\] contradicting the assumption that $(G,S,H_1,H_2)$ is a counterexample.
\end{proof}

\bibliographystyle{amsplain}
\bibliography{repcoh}

\end{document}